%% file: verifyingClosedManifolds.tex
\pdfoutput=1 

\newcommand\forNmd[1]{#1} \newcommand\forLMS[1]{}

\forNmd{
\documentclass[utopia,fullpage]{nmd/article}

\newcommand{\negSpace}{\!}

\setlength{\parskip}{5pt}
\setlength{\parindent}{0pt}
}

\forLMS{
\documentclass{lms}

\usepackage{url}
\usepackage{color}
\usepackage{mathtools}
\newcommand\C{\mathbb C}
\newcommand\R{\mathbb R}
\newcommand\Q{\mathbb Q}
\renewcommand\H{\mathbb H}
\newcommand\Isom{\mathrm{Isom}}

\newtheorem{theorem}{Theorem}[section]
\newtheorem{lemma}[theorem]{Lemma}
\newtheorem{definition}[theorem]{Definition}
\newtheorem{example}[theorem]{Example}
\newtheorem{conjecture}[theorem]{Conjecture}
\newtheorem{remark}[theorem]{Remark}

\newcommand{\negSpace}{}

\newcommand{\urladdr}[1]{\url{#1}}
\newcommand{\subjclass}[1]{#1}

}

\usepackage{graphicx}
\usepackage{amsfonts}

\newcommand\myIm{\mathrm{Im}}

\newcommand\du[1]{\underline{\overline{#1}}}

\newcommand\myPGL[2]{\mathrm{PGL}_{#1}(#2)}
\newcommand\myPSL[2]{\mathrm{PSL}_{#1}(#2)}
\newcommand\mySL[2]{\mathrm{SL}_{#1}(#2)}
\newcommand\mySO[1]{\mathrm{SO}_{#1}}
\newcommand\trig{\mathcal{T}}
\newcommand\groupG{\mathcal{G}}
\newcommand\permS{\mathbf{S}}
\newcommand\permA{\mathbf{A}}

\newcommand\myExp[1]{\exp(#1)}
\newcommand\myI{\imath}
\newcommand\myJ{\jmath}

\newcommand\myPath{\Gamma}

\newcommand\edgeEqSol{\mathrm{Sol}}

\begin{document}

\title{Verified computations for\\ closed hyperbolic 3-manifolds}
\author{Matthias Goerner}
\email{enischte@gmail.com}
\urladdr{http://unhyperbolic.org/}


\begin{abstract}
Extending methods first used by Casson, we show how to verify a hyperbolic structure on a finite triangulation of a closed 3-manifold using interval arithmetic methods. A key ingredient is a new theoretical result (akin to a theorem by Neumann-Zagier and Moser for ideal triangulations upon which HIKMOT is based) showing that there is a redundancy among the edge equations if the edges avoid ``gimbal lock''. We successfully test the algorithm on known examples such as the orientable closed manifolds in the Hodgson-Weeks census and the bundle census by Bell. We also tackle a previously unsolved problem and determine all knots and links with up to 14 crossings that have a hyperbolic branched double cover.
\end{abstract}

\maketitle

\setcounter{tocdepth}{1}

\tableofcontents

\section{Introduction}

Up to isometry, a finite hyperbolic 3-simplex is determined by its $6$ edge parameters, by which we mean either the edge lengths $l_{ij}$ or the respective entries of its vertex Gram matrix $v_{ij}=-\cosh(l_{ij})$. Thus, an assignment of a parameter to each edge of a finite triangulation $\trig$\negSpace{} of a closed 3-manifold determines a hyperbolic structure for each $3$-simplex of $\trig\negSpace{}$. If certain conditions are fulfilled, the hyperbolic structures on the individual simplices are compatible and form a hyperbolic structure on the manifold (see \cite{heardThesis} and Section~\ref{sec:hypStru}).

Existing software (such as Casson's Geo \cite{casson:geo} and Heard's Orb \cite{orb}) finds a numerical approximation for the edge parameters using Newton's method and reports whether the necessary equations are fulfilled within an error smaller than a certain $\varepsilon$. This suggests but does not prove hyperbolicity. The aim of this paper is to describe how to take such a numerical approximation and rigorously prove hyperbolicity by giving real intervals that are verified to contain a solution to all the necessary equations and inequalities. An algorithm either returning such intervals or (conservatively) reporting failure is described in Section~\ref{sec:algo}.
The algorithm is a hyperbolicity verification procedure but not a hyperbolicity decision procedure since its failure just means that the given candidate approximation was not close enough to a hyperbolic structure or needs to be perturbed to avoid ``gimbal lock'' (explained below). An implementation of this algorithm is available at \cite{veriClosedRepo}. Therefore, this paper is achieving for finite triangulations what Hoffman, Ichihara, Kashiwagi, Masai, Oishi, and Takayasu \cite{hikmot} did for ideal triangulations (HIKMOT's functionality has been integrated into SnapPy \cite{SnapPy} by the author since version~2.3). 

This is motivated by applications that benefit from using geometric finite triangulations in place of geometric spun triangulation. In particular, such applications no longer need to overcome the incompleteness locus. An example is the generation of cohomology fractals for a closed hyperbolic 3-manifold. As shown in  \cite[Figure~8.7]{cohomologyFractals}, the incompleteness locus produces artifacts in the raytraced image of a cohomology fractal which simply disappear when using finite triangulations instead. Another example is the algorithm proposed in \cite{HHGT} to rigorously compute the length spectrum for a hyperbolic 3-manifold. This algorithm requires tiling $\H^3$ with translates of a fundamental domain to cover a ball of specified radius and thus would not work if there is incompleteness locus.

An even more basic motivation is proving hyperbolicity of a closed 3-manifold by finding a triangulation admitting a geometric structure. Restricting ourselves to just spun triangulations introduces a bottleneck. For example, the obvious spun triangulation of a closed census manifold such as \texttt{m135(1,3)} can fail to be geometric. Thus finding a geometric spun triangulation requires drilling and filling (or, in other words, finding a different closed geodesic $\gamma$ such that there is a geometric triangulation spun about $\gamma$). Even worse, some hyperbolic 3-manifolds such as \texttt{m007(3,1)} seem to lack any geometric spun triangulation unless we pass to a cover\footnote{In not yet published work, Maria Trnkova has proven that there is no geometric spun triangulation of \texttt{m007(3,1)} with a small number of tetrahedra.}. Note that a geometric spun triangulation is known for every orientable closed manifold in the SnapPy census (except for \texttt{m007(3,1)} where a 3-fold cover is needed), see \cite{hikmot}. However, the process of finding such (covers admitting)  geometric spun triangulations can be tedious and is not known to be possible in general. Furthermore, passing to a cover can complicate applications such as computing the length spectrum.

Potential future work might generalize the techniques of this paper to Heard's work \cite{heardThesis} on 3-orbifolds and Frigerio and Petronio's work \cite{frigPet} on 3-manifolds with geodesic boundary. To find hyperbolic structures on these, Heard's program Orb \cite{orb} uses triangulations with finite as well as ideal and ``hyperinfinite'' vertices. Note that some of the theory in this paper also carries over to spherical and Euclidean geometry and might generalize to yield methods for verifying spherical or Euclidean structures on finite triangulations.

Like \cite{hikmot}, we use interval arithmetic methods such as the interval Newton method or the Krawczyk test. These methods can only show the existence of a solution to a system of equations if the Jacobian matrix is invertible near that solution. If the Jacobian fails to be invertible, these methods can only show the existence of a solution to a subset of the equations. This applies to the edge equations whether we are solving for shapes in the cusped case or for edge lengths in the closed case. Hence, in both cases, we need an additional result showing that there is a redundancy among the edge equations such that solving a suitable subset of them is sufficient. For ideal triangulations, this result is due to Neumann-Zagier \cite{NeumannZagier,NeumannComb} and Moser \cite{moser} (see Appendix). For finite triangulations, we derive such a result in this paper.

Note that while we actually have exactly as many variables as equations in the case of finite triangulations (namely, one per edge), the Jacobian of this system of equations has a kernel at a solution corresponding to a hyperbolic structure.
This is because we can move each individual finite vertex of a triangulation in the hyperbolic manifold and obtain a whole family of solutions (see Theorem~\ref{thm:hypStructAreSubmanifold}). Thus, we need to use a two-step strategy to verify a hyperbolic structure: First, we drop some edge equations and fix an equal number of edge parameters such that we can apply interval arithmetic methods to find intervals verified to contain a solution to the subsystem of equations we kept. Next, we show that this solution is also a solution to the equations we dropped earlier and thus that the intervals for the edge parameters contain a point giving a hyperbolic structure. Interval arithmetic can verify that the error of the dropped equations is small and we will show that if the dropped equations are fulfilled approximately, then they are fulfilled exactly provided that a certain condition we call ``gimbal lock'' is avoided.

To define gimbal lock, we will look at the complex of doubly-truncated simplices associated to the triangulation and an assignment of $\myPGL{2}{\C}$-matrices to the edges of the complex computed from the edge parameters (see Section~\ref{sec:cocycles}). The cocycle condition says that the matrices on the edges of a polygon must multiply to the identity. Since a subset of the edge equations is known to be fulfilled, the cocycle condition is known to hold for some polygons but not necessarily for others. The goal is to show that it holds for all polygons so that we get a $\myPGL{2}{\C}$-representation of the fundamental group (see Section~\ref{sec:vertexCocycles} and \ref{sec:extendCocycles} and examples in Section~\ref{section:examples}). Roughly speaking, the idea is that if the product of three small rotations about three axes in generic position is the identity, then each rotation must be the identity. Inspired by the mechanical device called gimbal (see Figure~\ref{fig:gimbal}), we say that we avoid ``gimbal lock'': if a gimbal is not in its locked position, then we can apply any small rotation to the inner-most ring, or equivalently, if we fix the inner-most ring, none of the other rings can be turned.

We describe the resulting algorithm to obtain real intervals in Section~\ref{sec:algo}. The algorithm is effective and able to verify a hyperbolic structure on all 36093 closed orientable manifolds in the Hodgson-Weeks census \cite{hwcensus} and in the census bundle by Bell \cite{bellBundleCensus}, see Section~\ref{sec:results}. 
Branched double covers of knots or links (or more precisely: double covers of $S^3$ branched over a knot or link) provide a good class of test cases since finding a geometric spun-triangulation of some of them can be challenging. Using finite triangulations instead, we are able to prove the following new result:
\begin{theorem} \label{thm:hyperbolicBranchedDoubleCovers}
Out of the 313230 knots with up to 15 crossings (not including the unknot), exactly 193839 have a hyperbolic branched double cover.\\
Out of the 120573 links with up to 14 crossings (with at least two components), exactly 37709 have a hyperbolic branched double cover.
\end{theorem}

Section~\ref{sec:discussion} concludes with a conjecture that implies that a hyperbolic structure on a finite triangulation can always be perturbed so that the algorithm can verify it.

The appendix in Section~\ref{sec:hikmotGap} points out a gap in the argument (but not the algorithm) of the HIKMOT paper.

\section*{Acknowledgements}


The author wishes to thank Marc Culler, Nathan Dunfield, Damian Heard, Neil Hoffman, and Christian Zickert for helpful discussions. Special thanks goes to Nathan Dunfield for finding some mistakes and providing some of the examples and some Python scripts for Regina.


\section{Hyperbolic structures on finite triangulations} \label{sec:hypStru}

Consider the isometry class of a positively oriented, finite geodesic simplex $\Delta$ with vertices labeled $0, \dots, 3$ in $\H^3$. We briefly review the relationship of the edge lengths and the angles of $\Delta$ following \cite{heardThesis} with one difference though: we use a slightly simpler definition for the vertex Gram matrix $G$ where all diagonal entries are $-1$ since we are not interested in generalized simplices here. To be consistent with the vertex labels, we $0$-index the rows and columns of a matrix (so $m_{00}$ denotes the top left-most entry). 

\begin{figure}[htb]
\begin{minipage}{\textwidth}
\begin{center}
\begin{minipage}{0.32\textwidth}
\begin{center}
\scalebox{1.0}{
\input{figures_gen/anglesInTet.tex}
}
\end{center}
\end{minipage}
\begin{minipage}{0.32\textwidth}
\begin{center}
\scalebox{0.9}{
\input{figures_gen/doublyTruncatedCocycle.tex}}
\end{center}
\end{minipage}
\begin{minipage}{0.32\textwidth}
\begin{center}
\scalebox{0.9}{
\input{figures_gen/doublyTruncatedCocycle_Prism.tex}}
\end{center}
\end{minipage}\\ 
\begin{minipage}[t]{0.36\textwidth}
\caption{Angles of simplex.} \label{fig:anglesInTet}
\end{minipage}
\begin{minipage}[t]{0.35\textwidth}
\caption{A doubly truncated simplex $\overline{\Delta}$, also known as permutahedron.}\label{fig:DoublyTruncated} \end{minipage}
\begin{minipage}[t]{0.26\textwidth}
\caption{A prism.}\label{fig:Prism}
\end{minipage}
\end{center}
\end{minipage}
\end{figure}

Let $l_{ij}$ denote the length of the edge between vertex $i$ and $j$. The vertex Gram matrix $G$ associated to the simplex is the symmetric $4\times 4$-matrix with entries $$v_{ij}=-\cosh(l_{ij}).$$ The edge lengths as well as the vertex Gram matrix uniquely determine the isometry class of the simplex. Let $c_{ij}$ denote the respective cofactor of $G$ which is given by
$$c_{ij}=(-1)^{i+j} \det(G_{ij}), $$
where $G_{ij}$ is obtained by deleting the $i$-th row and $j$-th column. The dihedral angle between face $i$ and $j$ and the angle at vertex $i$ of the triangle $ijk$ (derived from the law of cosines) are then given by (also see Figure~\ref{fig:anglesInTet}):
\begin{equation}
\theta_{ij}=\arccos\left(\frac{c_{ij}}{\sqrt{c_{ii}c_{jj}}}\right) \quad\mbox{and}\quad \eta_{i,jk}=\arccos\left(\frac{v_{ij}v_{ik}+v_{jk}}{\sqrt{v_{ij}^2-1}\sqrt{v_{ik}^2-1}}\right). \label{eqn:dihedral}
\end{equation}

\begin{definition}
Let $G$ be a real symmetric $4\times 4$-matrix with $-1$ on the diagonal. We say that $G$ is realized if $G$ is the vertex Gram matrix of some finite, non-flat simplex.
\end{definition}

The following theorem is a special case of \cite[Theorem~1.5]{heardThesis} (also compare to \cite[Theorem~7.2.2]{ratcliffe:hyp}):
\begin{lemma}
$G$ is realized if and only if
\begin{enumerate}
\item $G$ has one negative and three positive eigenvalues (which is equivalent to the characteristic polynomial $p_G(x)=\det(x I - G)=x^4+4x^3+a_2x^2+a_1x+a_0$ having coefficients $a_2 < 0$, $a_1>0$, $a_0<0$ by the Budan-Fourier theorem \cite{algRealGeom}),
\item $c_{ii}<0$ for all $i$, and\label{item:condNegAdj}
\item $c^2_{ij} < c_{ii} c_{jj}$ for all $i$ and $j$,\label{item:wellDefinedCos}
\end{enumerate}
where $c_{ij}$ denotes the respective cofactor of $G$. \label{lemma:singleGeomSimp}
\end{lemma}

Let $\trig$\negSpace{} be an oriented, finite 3-dimensional triangulation (i.e., all vertex links are $2$-spheres) and let $E(\trig)$ denote the set of edges of $\trig$\negSpace{}. 
Assume we have an assignment of a length $l_e>0$, or equivalently, a parameter $\nu_e < -1$ to each $e\in E(\trig)$ where the two are related by the formula $\nu_e=-\cosh(l_e)$. This induces a symmetric $4\times 4$-matrix $G_\Delta$ for each simplex $\Delta$ of $\trig$\negSpace{} where $v_{ii}=-1$ and $v_{ij}=\nu_e$ if $i\not=j$ and the edge of $\Delta$ from vertex $i$ to $j$ is incident to $e$. Let $\Theta_e$ denote the sum of all dihedral angles $\theta_{ij}$ incident to the edge $e$ of $\trig$\negSpace{}.
 We can now use \cite[Lemma~2.4]{heardThesis} to check whether this assignment yields a hyperbolic structure on $\trig$:
\begin{theorem}
An assignment of a parameter $\nu_e<-1$ to each edge $e$ of an oriented, finite triangulation $\trig$\negSpace{} induces a hyperbolic structure on $\trig$\negSpace{} if
\begin{enumerate}
\item \label{main:condA} each matrix $G_\Delta$ is realized (i.e., fulfills the conditions of Lemma~\ref{lemma:singleGeomSimp}) and
\item \label{main:condB} $\Theta_e=2\pi$ for every $e\in E(\trig)$.
\end{enumerate} \label{thm:hypStruct}
\end{theorem} 

\section{Cocycles} \label{sec:cocycles}

We will describe how to compute a representation\footnote{For consistency on how arrows in a category compose, the loop traversing the loop $a$ first and the loop $b$ second is denoted by $ba$ in $\pi_1(\trig)$.} $\pi_1(\trig)\to\myPGL{2}{\C}$ from an assignment of parameters $\nu_e$ as in Theorem~\ref{thm:hypStruct} using cocycles inspired by \cite[Section~9]{higherGluingEqns} as follows:
\begin{definition}
Let $\groupG$ be a group and $X$ be a space with a polyhedral decomposition.
A $\groupG$-cocycle on $X$ is an assignment of elements in $\groupG$ to the oriented edges of $X$ such that the product around each face is the identity and such that reversing the orientation of an edge replaces the labeling by its inverse.
\end{definition}

All cocycles in this section are $\myPGL{2}{\C}$-cocycles.

Given a matrix $G$ fulfilling the conditions in Lemma~\ref{lemma:singleGeomSimp}, we will construct a cocycle on the doubly truncated simplex $\overline{\Delta}$ coming from a simplex $\Delta$, see Figure~\ref{fig:DoublyTruncated}. We index a vertex $v$ of $\overline{\Delta}$ by the permutation $\sigma\in \permS_4$ such that the vertex of $\Delta$ closest to $v$ is $\sigma(0)$, the edge of $\Delta$ closest to $v$ is $\sigma(0)\sigma(1)$, and the face of $\Delta$ closest to $v$ is $\sigma(0)\sigma(1)\sigma(2)$. We label an oriented long, middle, or short edge of $\overline{\Delta}$ by $\alpha^{\sigma(0)\sigma(1)\sigma(2)}, \beta^{\sigma(0)\sigma(1)\sigma(2)},$ respectively $\gamma^{\sigma(0)\sigma(1)\sigma(2)}$ if it starts at the vertex indexed by $\sigma$. Let us introduce the notion of standard position to define the $\myPGL{2}{\C}$-matrices assigned to these edges.

\begin{figure}
\begin{center}
\scalebox{0.9}{
\input{figures_gen/tetInStandardPosition.tex}}
\end{center}
\caption{A simplex in different standard positions.} \label{fig:standardPos}
\end{figure}
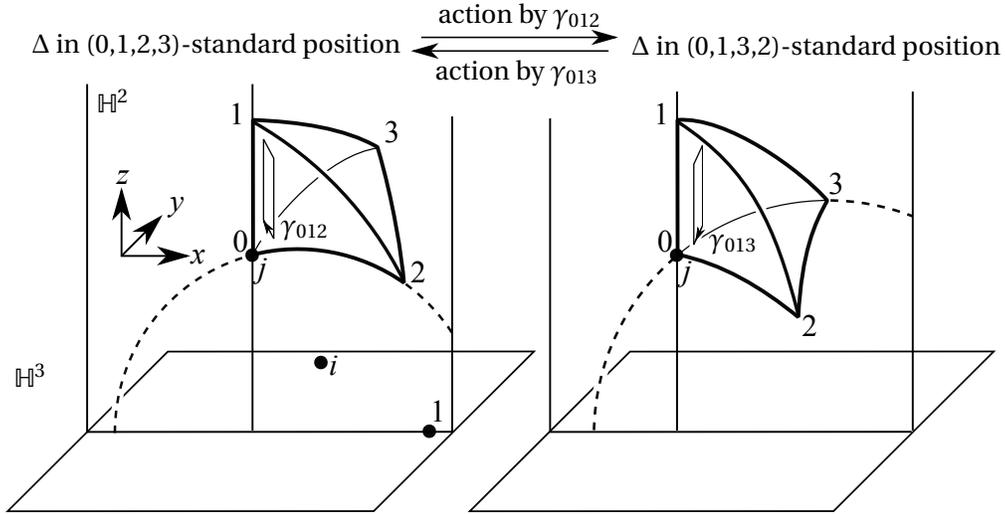

We think of the upper half space model of hyperbolic 3-space  as a subset $\H^3=\{w = z+t\myJ : t>0\}$ of the quaternions. Recall that $\mathrm{Isom}^+(\H^3)\cong\myPGL{2}{\C} \cong\myPSL{2}{\C}$ where the action of a $\mySL{2}{\C}$-matrix on $\H^3$ is given by
$$\left(\begin{array}{cc}a &b\\ c&d\end{array}\right)\mapsto \left( w \mapsto (aw+b)\cdot(cw+d)^{-1} \right).$$

\begin{definition} \label{def:standardPosSimplex}
We say that a positively oriented, finite simplex $\Delta$ is in $\sigma$-standard position where $\sigma\in \permS_4$ if vertex
\begin{itemize}
\item $\sigma(0)$ is at $\myJ$,
\item $\sigma(1)$ at $t\myJ$ with $t>1$,
\item $\sigma(2)$ at $a+t\myJ$ with $a>0$ and
\item $\sigma(3)$ at $z+t\myJ$ with $\myIm(z)>0$ if $\sigma$ is even and $\myIm(z)<0$ otherwise.
\end{itemize}
\end{definition}
An example of this is shown in Figure~\ref{fig:standardPos}.
Geometrically, the motivation for this definition is that two faces of two (not necessarily distinct) simplices line up in $\H^3$ if the respective edge lengths match and the two simplices are in the respective standard positions. More precisely, let $f_1\in\{0,1,2,3\}$ be a face of the simplex $\Delta_1$ and $f_2\in\{0,1,2,3\}$ of $\Delta_2$. Let $\sigma\in \permS_4\setminus \permA_4$ with $\sigma(f_1)=f_2$ be a pairing of the two faces. If the edge lengths match under this pairing, then the faces match if each $\Delta_k$ is in $\sigma_k$-standard position where $\sigma_1$ is any permutation with $\sigma_1(3)=f_1$ and $\sigma_2=\sigma\circ \sigma_1$.

This definition also gives us a cocycle as follows (also see Figure~\ref{fig:standardPos}):

\begin{definition}
Consider the isometry class of a finite simplex $\Delta$. Given an oriented edge $e$ of $\overline{\Delta}$, let $\sigma$ and $\sigma'$ be the permutations that index the vertex where $e$ starts and, respectively, ends. The natural cocycle on $\overline{\Delta}$ is the cocycle assigning to each edge $e$ the $\myPGL{2}{\C}$-matrix taking $\Delta$ from $\sigma$-standard position to $\sigma'$-standard position.
\end{definition}

Note that we can identify Euclidean 3-vector space isometrically with the tangent space of a point in $\H^3$ such that the tangents corresponding to the $x$-, $y$-, and $z$-axis are parallel to the real line, the imaginary line, respectively, the line $t\myJ$ (see Figure~\ref{fig:standardPos}). Thus, we can associate a $\mySO{3}$-matrix to an element in $\Isom^+(\H^3)$ fixing a (finite) point of $\H^3$. Let
$$R_\omega = \left(\begin{array}{ccc} \cos\omega & -\sin\omega & 0 \\ \sin\omega & \phantom{-}\cos\omega &  0\\0 &0 & 1\end{array}\right)$$
be the rotation about the $z$-axis by the angle $\omega$.

\begin{lemma}
The natural cocycle on $\overline{\Delta}$ can be computed from the vertex Gram matrix $G$ as follows (apply even permutations $\sigma\in \permA_4$ to obtain labels for all edges):
$$\alpha^{120}=\alpha^{210}=\alpha^{123}=\alpha^{213}=\left(\begin{array}{cc}0  & \sqrt{v_{12}^2-1} - v_{12}\\1 & 0 \end{array}\right),$$
$$\beta^{123}=\beta^{132}=\left(\begin{array}{cc}-\cos(\eta_{1,32}/2) & \sin(\eta_{1,32}/2)\\ \phantom{-}\sin(\eta_{1,32}/2) & \cos(\eta_{1,32}/2)\end{array}\right),$$
$$\gamma^{123}=\left(\gamma^{120}\right)^{-1}=\gamma^{210}=(\gamma^{213})^{-1}=\left(\begin{array}{cc}\myExp{\myI \theta_{03}} & 0 \\ 0 & 1\end{array}\right).$$
Note that $\beta^{123}$ and $\gamma^{123}$ fix the point $\myJ\in\H^3$ and the associated $\mySO{3}$-matrices are:
$$\left(\begin{array}{ccc}
-\cos\eta_{1,32} & \phantom{-}0 & \sin \eta_{1,32} \\
 0 & -1 & 0 \\
  \phantom{-}\sin\eta_{1,32} & \phantom{-}0 & \cos\eta_{1,32} \end{array}\right)\quad\mbox{and}\quad R_{\theta_{03}}.
$$
\label{lemma:coycleAssignment}
\end{lemma}

\begin{proof}
$\alpha^{120}$ is an involution exchanging the points $\myJ$ and $\myExp{l_{12}}\myJ$. The associated $\C P^1$-automorphism is of the form $z\mapsto x/z$ and exchanges $1$ and $\myExp{l_{12}}$. An elementary calculation gives the value for $x$.\\
Consider the hyperbolic plane $\H^2=\{x+t\myJ : x\in \R, t>0\}\subset \H^3$. $\beta^{123}$ is the composition of the rotation of $\H^2$ about $\myJ$ by $\eta_{1,32}$ with the involution fixing the line $t\myJ$ pointwise. We obtain the $\myPSL{2}{\R}$-matrix for the rotation of $\H^2$ by conjugating the rotation of the unit disk by $\eta_{1,32}$ with the matrix taking the upper half plane model to the Poincare disk model of hyperbolic 2-space:
$$\left(\begin{array}{cc}\myI & 1\\ 1 & \myI\end{array}\right)^{-1} \cdot \left(\begin{array}{cc} \myExp{\myI \eta_{1,32}/2} & 0 \\ 0 & \myExp{-\myI \eta_{1,32}/2}\end{array}\right)\cdot \left(\begin{array}{cc}\myI & 1\\ 1 & \myI\end{array}\right)$$
\end{proof}

Given an oriented triangulation $\trig$\negSpace{}, let $\overline{\trig}$ be the complex obtained by replacing each simplex $\Delta$ by the double truncated simplex $\overline{\Delta}$. The long and short edges of $\overline{\trig}$ about an edge $e$ of $\trig$\negSpace{} form a prism, see Figure~\ref{fig:Prism}. Given an assignment of edge parameters $\nu_e$ for $\trig$\negSpace{}, this prism about an edge $e\in E(\trig)$ is a cocycle if and only if the short edges compose to the identity, which is equivalent to $\Theta_e$ being a multiple of $2\pi$. Let $\hat\trig$\negSpace{} denote the complex $\overline{\trig}\cup\mbox{Prisms}$.

\begin{theorem} \label{thm:cocycleExtendingToTrig}
Consider an assignment of a parameter $\nu_e<-1$ to each edge of an oriented, finite triangulation $\trig$\negSpace{}. 
\begin{enumerate}
\item \label{cocycCondA} If each matrix $G_\Delta$ is realized, we obtain a natural cocycle on $\overline{\trig}$ and, thus, a representation $\pi_1(\overline{\trig})\to\myPGL{2}{\C}$ (up to conjugation, unless we pick a vertex of $\overline{\trig}$ as basepoint).
\item \label{cocycCondB} If, furthermore, $\Theta_e$ is a multiple of $2 \pi$ for every $e\in E(\trig)$, the cocycle extends to $\hat{\trig}$\negSpace{} and, thus, yields a representation of $\pi_1(\trig)\to\myPGL{2}{\C}$.
\item \label{cocycCondC} If, furthermore, $\Theta_e = 2\pi$ for every $e\in E(\trig)$, the representation is giving a hyperbolic structure on $\trig$\negSpace{}. \label{thm:cocycle}
\end{enumerate}
 \end{theorem}
 
 \begin{proof}
Note that $\alpha$ and $\beta$ in Lemma~\ref{lemma:coycleAssignment} are involutions and only involve the parameters $v_{ij}$ on the edges of the triangle containing the respective $\alpha$ and $\beta$. Hence, the matrices on two big hexagons on two doubly-truncated simplices are compatible and the hexagons can be identified if the edge parameters on the respective triangles of the corresponding tetrahedra match. This proves \eqref{cocycCondA}.\\
\eqref{cocycCondB} follows from the above comment about the prisms being cocycles and the fact that $\hat{\trig}$ differs from $\trig$\negSpace{} only by a set of 3-balls which do not change $\pi_1$.\\
\eqref{cocycCondC} is just restating Theorem~\ref{thm:hypStruct}.
 \end{proof}
 
 
 \section{Extending cocycles on genus 0 surfaces} \label{sec:vertexCocycles}

Let us define a punctured topological polyhedron which we will use as model for a ``vertex link'' of $\hat\trig$\negSpace{} when removing some prisms.
 
\begin{definition} \label{def:puncTopPolyhed}
 Let $L$ be a topological polyhedron, i.e., a decomposition of an oriented $2$-sphere into polygons. Let $P_1, \dots, P_p$ be selected two-cells of $L$ such that
 $\partial P_i \cap \partial P_j \not =\emptyset$ for all $i\not=j$. We call $(L,L\setminus \bigcup P_l)$ a punctured topological polyhedron.
\end{definition}

The complex $L\setminus \bigcup P_l$ will often be equipped with the following kind of cocycle.
 
\begin{definition}
Let $X$ be a surface with boundary and with a decomposition into polygons. A $(\mySO{3}, \mySO{2})$-cocycle on $X$ is a $\mySO{3}$-cocycle where edges in the boundary $\partial X$ are labeled by rotations $R_\omega\in\mathrm{Im}(\mySO{2}\hookrightarrow\mySO{3})$ about the $z$-axis.
\end{definition}
 
The goal of this section is to give a criterion when such a cocycle on $L\setminus \bigcup P_l$ extends to a $\mySO{3}$-cocycle on $L$.
 
 \begin{definition}
 Let $X$ be a $2$-complex. An edge-loop $\myPath$ is a word in the oriented edges of $X$ such that the end of one edge coincides with start of the next edge (when reading the word from right to left and cyclically).
 \end{definition}
 
Note that by parametrizing the oriented edges of $X$ and concatenating them in the order given by $\myPath$, we obtain a geometric realization of $\myPath$ that is a (based) topological loop in $X$. Also note that a $\groupG$-cocycle on $X$ assigns a value in $\groupG$ to an edge-loop $\myPath$ obtained by multiplying the labels of the oriented edges in the respective order. The value assigned to $\myPath$ depends only on the homotopy type of (the geometric realization of) $\myPath$ and, in particular, is trivial if the loop is contractible in $X$.

\begin{definition} \label{def:gimbalLoop}
Let $(L,L\setminus \bigcup P_l)$ be a punctured topological polyhedron. A gimbal loop $\myPath$ is a word in the oriented edges of $L$ and the selected two cells $P_1, \dots, P_p$ such that
\begin{itemize}
\item each $P_1, \dots, P_p$ is contained in $\myPath$ exactly once,
\item each $P_l$ in $\myPath$ is preceeded (when reading the word from right to left and cyclically) by an edge $e_j$ such that the endpoint of $e_j$ is a vertex of $\partial P_l$, and
\item dropping all $P_1, \dots, P_p$ from $\myPath$ results in an edge loop. This edge loop bounds a disk in $L\setminus\bigcup P_l$. The interior of the disk embeds into $L\setminus\bigcup P_l$ matching the orientation of $L$.

\end{itemize}
\end{definition}

\begin{example} \label{example:GimbalLoop}
Figure~\ref{fig:gimbalLoop} shows an example of a gimbal loop $\myPath$ for $p=3$. The corresponding word is $e_6 P_2 e_5 e_4 P_1 e_3 e_2 P_3 e_1.$
\end{example}

\begin{figure}[htb]
 \begin{minipage}{\textwidth}
  \begin{center}
   \begin{minipage}{0.49\textwidth}
    \begin{center}
     \scalebox{0.25}{
      \input{figures_gen/gimbalLoop.tex}
     }
    \end{center}
   \end{minipage}
   \begin{minipage}{0.49\textwidth}
    \begin{center}
     \scalebox{0.25}{
      \input{figures_gen/gimbalLoopExtended.tex}
     }
    \end{center}
   \end{minipage}%
\\ \vspace{0.4cm}
   \begin{minipage}[t]{0.49\textwidth}
    \caption{The edge loop obtained when dropping the $P_l$ from a gimbal loop $\myPath$.} \label{fig:gimbalLoop}
   \end{minipage}
   \begin{minipage}[t]{0.49\textwidth}
    \caption{The edge loop obtained when replacing $P_l$ by $\partial P_l$ in a gimbal loop $\myPath$.} \label{fig:splicedGimbalLoop}
   \end{minipage}
  \end{center}
 \end{minipage}
\end{figure}

\begin{definition} \label{def:defGimbalFunction}
Let $(L,L\setminus \bigcup P_l)$ be a punctured topological polyhedron and $\myPath$ be a gimbal loop. Fix a $(\mySO{3},\mySO{2})$-cocycle on $L\setminus \bigcup P_l$. Assume we are given numbers $(d_1,\dots,d_p)$. Recall that the cocycle assigns a $\mySO{3}$-matrix to every edge in $\myPath$. Assign to each $P_l$ in $\myPath$ the rotation $R_{d_l}\in\mathrm{Im}(\mySO{2}\hookrightarrow\mySO{3})$. We call the product of the matrices assigned to the letters in $\myPath$ the gimbal matrix $m_\myPath(d_1,\dots,d_p)$. Furthermore, we call the function
$$g_\myPath:\R^p\to\R^3,\quad (d_1,\dots, d_p)\mapsto\left(m_\myPath(d_1,\dots,d_p)_{01},m_\myPath(d_1,\dots,d_p)_{02},m_\myPath(d_1,\dots,d_p)_{12}\right)$$
assigning the upper triangular entries of the gimbal matrix the gimbal function.
\end{definition}

Recall that a $(\mySO{3},\mySO{2})$-cocycle on $L\setminus \bigcup P_l$ assigns a rotation $R_\omega\in\mathrm{Im}(\mySO{2}\hookrightarrow\mySO{3})$ to each edge in $\partial P_l$ (with orientation induced from the orientation of $P_l\subset L$) where we pick $\omega\in(-\pi,\pi]$. Let $\delta_l$ denote the sum of all the angles $\omega$ over the edges of $\partial P_l$. For example, $\delta_2$ in Figure~\ref{fig:gimbalLoop} is the sum of the angles $\omega$ associated to the edges $e^{-1}_{6}, e_{19}, e_{18}, e_{17}, e_{16},$ and $e_{15}.$

\begin{lemma}
Using the same setup as in Definition~\ref{def:defGimbalFunction}, we have
$g_\myPath(2\pi,\dots,2\pi)=g_\myPath(\delta_1,\dots,\delta_p)=(0,0,0)$. \label{lemma:gimbalMatrixId}
\end{lemma}

\begin{proof}
Assume each $d_l=2\pi$. Then each $R_{d_l}$ is the identity and thus can be dropped from the gimbal matrix $m_\myPath(2\pi,\dots,2\pi)$. Thus the gimbal matrix is equal to the matrix the cocycle assigns to the edge loop obtained when dropping the $P_i$. This edge loop is contractible. Thus the gimbal matrix is the identity and $g_\myPath(2\pi,\dots,2\pi)=(0,0,0)$.\\
The gimbal loop $\myPath$ can also be turned into an edge loop by replacing each $P_l$ in $\myPath$ by a word in the oriented edges in $\partial P_l$. Here we use the orientation of $\partial P_l$ induced from the $P_l\subset L$ and start traversing $\partial P_l$ at the endpoint of the preceding edge. Figure~\ref{fig:splicedGimbalLoop} shows this edge loop for Example~\ref{example:GimbalLoop}. The corresponding word is:
$$e_6(e_6^{-1}e_{19}e_{18}e_{17}e_{16}e_{15})e_5e_4(e_4^{-1}e_{14}e_{13}e_{12}e_{11})e_3e_2(e_2^{-1}e_{10}e_9e_8e_7)e_1.$$
Note that by definition of $\delta_l$, the matrix assigned to this edge loop by the cocycle is equal to the gimbal matrix $m_\myPath(\delta_1,\dots,\delta_p)$. This edge loop is again contractible and thus $g_\myPath(\delta_1,\dots,\delta_p)=(0,0,0)$. To see that the edge loop is contractible, note that a genus 0 surface with $p$ boundary components can be obtained by attaching a 2-cell to a 1-complex consisting of $p$ loops and $p$ arcs connecting the loops to common base point. The edge looped can be homotoped into this form since it bounds a disk with interior embedding into $L\setminus \cup P_l$, see Figure~\ref{fig:gimbalLoopSkeleton}.%
\begin{figure}[h]
\begin{minipage}{\textwidth}
\begin{center}
\begin{minipage}{0.49\textwidth}
\begin{center}
\scalebox{0.25}{
\input{figures_gen/gimbalLoopSkeleton.tex}
}
\end{center}
\end{minipage}
\begin{minipage}{0.49\textwidth}
\begin{center}
\scalebox{1.0}{
\includegraphics[height=7cm]{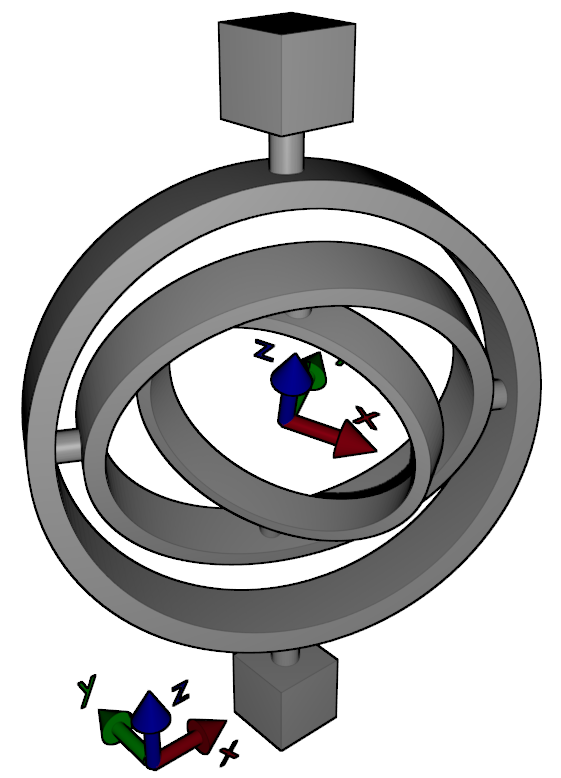}
}
\end{center}
\end{minipage}\\
\begin{minipage}{0.49\textwidth}
\caption{Homotoping the edge loop obtained from gimbal loop $\myPath$ such that it bounds a disk.} \label{fig:gimbalLoopSkeleton}
\end{minipage}
\begin{minipage}{0.49\textwidth}
\caption{A gimbal.} \label{fig:gimbal}
\end{minipage}
\end{center}
\end{minipage}
\end{figure}
\end{proof}

Let $Dg_\myPath:\R^p\to M(3\times p, \R)$ denote the Jacobian of $g_\myPath$ and $[Dg_\myPath(K)]$ be the interval closure of $Dg_\myPath(K)$, i.e., the smallest (axis-aligned) closed box containing $Dg_\myPath(K)$ when thinking of the matrix space $M(3\times p, \R)$ as $\R^{3p}$. We say that $[Dg_\myPath(K)]$ is invertible if every matrix in $[Dg_\myPath(K)]$ is invertible.

\begin{definition} \label{def:gimbalLockAvoided}
Using the same setup as in Definition~\ref{def:defGimbalFunction}, let $K\subset \R^p$ be a box. We say that $K$ avoids gimbal lock with respect to $\myPath$ if $[Dg_\myPath(K)]$ is invertible.
\end{definition}

\begin{remark}
The term gimbal lock is inspired by the mechanical device called gimbal or Cardan suspension used to achieve an arbitrary rotation in $\mySO{3}$, see Figure~\ref{fig:gimbal}. Letting $d_1, d_2$ and $d_3$ denote the Euler angles at the joints from the grounding boxes to the inner-most ring, the rotation achieved by the gimbal is given by the matrix
$$(d_1,d_2,d_3)\mapsto R_{d_1} \left(\begin{smallmatrix} &  & -1\\ & 1 & \\ 1 & & \end{smallmatrix}\right) R_{d_2} \left(\begin{smallmatrix} &  & 1\\ & 1 & \\ -1 & & \end{smallmatrix}\right) R_{d_3}$$
and a configuration of $(d_1, d_2, d_3)$ where this map does not have full rank is known as gimbal lock. Note the formal similarity of this product to the gimbal matrix $m_\myPath(d_1,\dots, d_p)$.
\end{remark}

The following lemma is not useful in the general setting, but illustrates the principle we will use later (also see Example~\ref{example:singleEdgeViolation}):

\begin{lemma}
Using the same setup as in Definition~\ref{def:defGimbalFunction}, let $K\subset \R^p$ be a box avoiding gimbal lock with respect to $\myPath$ such that $(2\pi, \dots, 2\pi) \in K$. If $(\delta_1,\dots,\delta_p)\in K$, then all $\delta_l=2\pi$ and the cocycle extends to $L$. \label{lemma:noGimbalImpliesAllEqnsVertexLink}
\end{lemma}

\begin{proof}
A standard result about the interval Newton method says that $g_\myPath$ is injective on $K$ since $[Dg_\myPath(K)]$ is invertible. Thus, Lemma~\ref{lemma:gimbalMatrixId} implies that $\delta_l=2\pi$.
\end{proof}

\section{Extending cocycles on triangulations} \label{sec:extendCocycles}

Let $\trig$ be a triangulation with an assignment of edge parameters $\nu_e<-1$ such that each $G_\Delta$ is realized (i.e., fulfills the conditions of Lemma~\ref{lemma:singleGeomSimp}). Let $E^\sim \cup E^==E(\trig)$ be a partition of the edges into two disjoint sets where $|E^\sim| = 3o$ with $o=|V(\trig)|$ being the number of vertices of $\trig$\negSpace{}. Assume that we know that $\Theta_e=2\pi$ for every edge $e\in E^=$, but we only know that $\Theta_e$ is close to $2\pi$ for $e\in E^\sim$. Thus, the $\myPGL{2}{\C}$-cocycle in Theorem~\ref{thm:cocycleExtendingToTrig} might only extend to $\hat\trig^==\hat\trig\setminus \mathrm{Prisms}(E^\sim)$, the complex where the prisms (including the top and bottom face) about the edges in $E^\sim$ have been removed. The goal of this section is to give a criterion that forces $\Theta_e=2\pi$ for all $e\in E(\trig)$, so that the conclusions of Theorem~\ref{thm:cocycleExtendingToTrig} apply and the edge parameters $\nu_e$ yield a hyperbolic structure for $\trig$\negSpace{}.

Let $V(\trig)=\{v_1,\dots,v_o\}$ be the vertices of $\trig$\negSpace{}. If we drop all all $\alpha$-edges and all $2$- and $3$-cells adjacent to any $\alpha$-edge from $\hat\trig$, the remaining complex has a connected component for each vertex $v_k$. We call such a connected component the ``vertex link'' of $v_k$ and denote it by $L_k$. Starting with $\hat\trig^=$ instead of $\hat\trig$, we similarly obtain a subcomplex $L^=_k\subset L_k$ for each $v_k$. Note that $(L_k,L^=_k)$ is a punctured topological polyhedron as in Definition~\ref{def:puncTopPolyhed} and the $\mySO{3}$-cocycle on $\hat\trig^=$ descends to a $(\mySO{3},\mySO{2})$-cocycle on each $L^=_k$ formed by the $\beta$- and $\gamma$-edges. Let us fix a gimbal loop $\myPath_k$ for each $(L_k,L^=_k)$, giving us gimbal functions $g_1,\dots, g_o$ as in Definition~\ref{def:defGimbalFunction}.

\begin{figure}
\begin{center}
\scalebox{0.7}{
\input{figures_gen/gimbalLoopTriangulation.tex}}
\end{center}
\caption{Gimbal loops for two ``vertex links'' (only some part of each is shown). The thick line is the edge in the triangulation.} \label{fig:gimbalLoopsInTriangulation}
\end{figure}
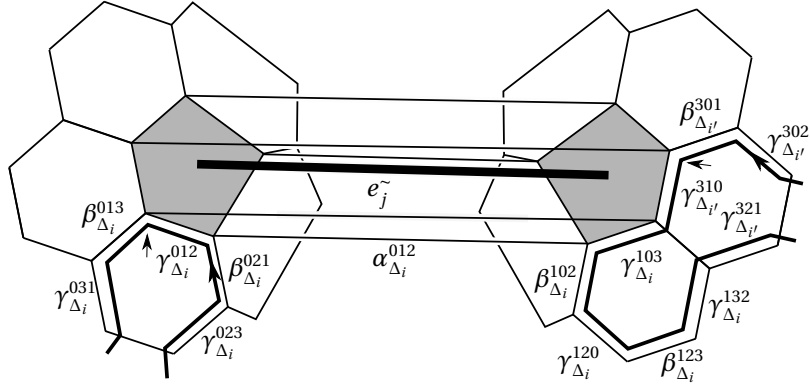

Let us label the edges in $E^\sim = \{e^\sim_1,\dots,e^\sim_{3o}\}$ and $E^==\{e^=_1,\dots, e^=_{m-3o}\}$. We again want to construct a gimbal function, but this time in one variable $T_j$ per edge $e^\sim_j\in E^\sim$ instead of a variable $d_{k,l}$ per removed polygon as in Section~\ref{sec:vertexCocycles}. Note that each edge $e^\sim_j\in E^\sim$ corresponds to two polygons $P_{k,l}$ and $P_{k',l'}$ that have been removed (from the ``vertex links'' $L_k$ and $L_{k'}$ in $\hat\trig$) to obtain $L^=_k$ and $L^=_{k'}$ (with $k$ and $k'$ not necessarily distinct), see Figure~\ref{fig:gimbalLoopsInTriangulation}.  We obtain a gimbal function in $(T_1,\dots, T_{3o})$ for each vertex by setting $d_{k,l}=d_{k',l'}=T_j$ in Definition~\ref{def:defGimbalFunction}. We combine these gimbal functions into a single one:

\begin{definition} \label{def:trigGimbalFunction}
Let $\trig$\negSpace{} be a triangulation with an assignment of edge parameters $\nu_e<-1$ such that each $G_\Delta$ fulfills the conditions of Lemma~\ref{lemma:singleGeomSimp}. Let $E^\sim = \{e^\sim_1,\dots,e^\sim_{3o}\}$ and $E^==\{e^=_{1},\dots, e^=_{m-3o}\}$ be a partition of $E(\trig)$ and fix gimbal loops $\myPath_1,\dots,\myPath_o$. The gimbal function is given by
$$g: \R^{3o} \to \R^{3 o}, \quad (T_1,\dots, T_{3o}) \mapsto \left(g_1(d_{1,1}, \dots, d_{1, p_1}), \dots, g_o(d_{o,1}, \dots, d_{o,p_o})\right).$$
A box $K\subset \R^q$ avoids gimbal lock if $[Dg(K)]$ is invertible.
\end{definition}

Note that the orientation on $(L_k,L^=_k)$ in the above definitions must be chosen such that the boundary of a small hexagon in $(L_k,L^=_k)$ contains a $\gamma^{\sigma(0)\sigma(1)\sigma(2)}$ for $\sigma\in \permA_4$.

\begin{example}
Figure~\ref{fig:gimbalLoopsInTriangulation} shows examples of gimbal loops for two ``vertex links'' of $\hat{\trig}^=$. The corresponding gimbal matrices would be given by
$$\cdots \gamma^{031}_{\Delta_i} \cdot
 \beta^{013}_{\Delta_i} \cdot 
 R_{T_j} \cdot
  \gamma^{012}_{\Delta_i} \cdot
   \beta^{021}_{\Delta_i}
     \cdots
      \quad\mbox{and}\quad
       \cdots
          \gamma^{103}_{\Delta_i}
          \cdot
           \gamma^{310}_{\Delta_{i'}}
           \cdot
            R_{T_j} \cdot \beta^{301}_{\Delta_{i'}}
            \cdot
             \gamma^{302}_{\Delta_{i'}}\cdots .$$
The gimbal function $g$ for the triangulation is obtained by taking the upper triangular entries of the gimbal matrix for each vertex.
\end{example}

\begin{theorem} \label{thm:mainThm}
Let $\trig\negSpace{}$, $\nu_e$, $E^\sim$ and $E^=$, $\myPath_1, \dots, \myPath_o$ and $g$ be as in Definition~\ref{def:trigGimbalFunction}. Let $K\subset \R^{3o}$ be a box avoiding gimbal lock such that $(2\pi, \dots, 2\pi)\in K$. If
\begin{enumerate}
\item $(\Theta_{e^\sim_1},\dots,\Theta_{e^\sim_{3o}})\in K$ and
\item $\Theta_{e^=_1}=\dots=\Theta_{e^=_{m-3o}}=2\pi$
\end{enumerate}
then $\Theta_e=2\pi$ for all $e\in E(\trig)$, so the edge parameters $\nu_e$ yield a hyperbolic structure on $\trig$\negSpace{}.
\end{theorem}

\begin{proof}
By applying Lemma~\ref{lemma:gimbalMatrixId} to each vertex $v_1,\dots, v_o$, we see that $$g(2\pi,\dots,2\pi)=g(0,\dots,0)=(0,\dots,0).$$ Hence, the proof from Lemma~\ref{lemma:noGimbalImpliesAllEqnsVertexLink} applies here as well and we have $\Theta_{e^\sim_1}=\dots=\Theta_{e^\sim_{3o}}=2\pi$. Therefore, all conditions of Theorem~\ref{thm:cocycleExtendingToTrig} are fulfilled.
\end{proof}

\section{Examples of (non-)gimbal lock} \label{section:examples}

This section is giving examples where the condition $\Theta_e=2\pi$ from Theorem~\ref{thm:hypStruct} is known to be fulfilled for most but not all edges.
Let $\trig$ be a finite, orientable triangulation with $o$ vertices and $m$ edges. We partition the edges as
  $E^\sim=\{e^\sim_1,\dots,e^\sim_l\}$ and $E^==\{e^=_1,\dots,e^=_{m-l}\}$ depending on whether we know that $\Theta_e$ is close to or, respectively, exactly equal to $2\pi$.

 Geometrically, this yields a singular hyperbolic structure with cone singularities along the edges with $\Theta_{e^\sim_i}\not=2\pi$. Looking at a vertex $v$ connected to such an edge, its neighborhood is the hyperbolic cone of its link which is $S^2$ topologically but has a spherical cone structure different from the standard $S^2$. In this section, we extend Definition~\ref{def:trigGimbalFunction} to the case where $l=|E^\sim|$ might not be $3o$ and say that gimbal lock is avoided if $Dg$ has no kernel.

\begin{example} \label{example:singleEdgeViolation} Consider the case $l=1$ with $e^\sim_1$ connecting two distinct vertices $v_1$ and $v_2$. Assume $\Theta_{e^\sim_1}\not=2\pi$. The links of $v_1$ and $v_2$ would have a spherical cone structure with exactly one singularity. These do not exist, so $\Theta_{e'}=2\pi$. This also follows from Lemma~\ref{lemma:noGimbalImpliesAllEqnsVertexLink}.
\end{example}

\begin{example}[Gimbal lock]\label{ex:gimbalLockGeodesic} Let us start with a (non-singular) hyperbolic structure on $\trig$\negSpace{} such that the edges $e^\sim_1,\dots, e^\sim_l$ form a simple closed geodesic. Such a hyperbolic structure can often be deformed so that the geodesic becomes a cone singularity\footnote{For example, \texttt{Manifold("m003(-3.3,1.1)")} in SnapPy gives a geometric spun triangulation for the singular hyperbolic structure on the closed Weeks manifold \texttt{m003(-3,1)} with cone angle $2\pi/1.1$.}, i.e., such that $\Theta_{e^\sim_1}=\cdots=\Theta_{e^\sim_l}$ are close but not equal to $2\pi$. Note that the spherical structure of the link of a vertex $v_i$ connected to $e^\sim_i$ and $e^\sim_{i+1}$ gets deformed to have two conical singularities at antipodal points with equal cone angle. In the language of cocycles, pick a gimbal loop $\Gamma_i$ for $v_i$ as in Definition~\ref{def:gimbalLoop} such that the associated gimbal matrix is of the form 
\begin{equation} \label{eqn:gimbalMatrixAntipodal}
m_{\Gamma_i}=R_{\Theta_{e^\sim_i}}\cdot \prod\nolimits_1 \cdot R_{\Theta_{e^\sim_{i+1}}}\cdot \prod\nolimits_1^{-1} = R_{\Theta_{e^\sim_i}} R_{-\Theta_{e^\sim_{i+1}}}
\end{equation}
where $\prod_1\in\mySO{3}$ is a rotation by $\pi$ about an axis orthogonal to $z$. Then $m_{\Gamma_i}=\mathrm{Id}$ whenever $\Theta_{e^\sim_i}=\Theta_{e^\sim_{i+1}}$, so $Dg$ has a non-trivial kernel. If we add edges to $E^\sim$ such that $|E^\sim|=3o$, $Dg$ still has a non-trivial kernel, so we have gimbal lock and thus not a contradiction to Theorem~\ref{thm:mainThm}.
\end{example}

\begin{example}[Perturbing previous example to avoid gimbal lock] \label{ex:avoidGimbalLockGeodesic}
Consider the situation in the previous example but move one vertex, say $v_1$, slightly. Assume $\Theta_{e^\sim_1}\not=2\pi$. The link of $v_1$ would be a spherical cone structure with exactly two cone singularities having distance $<\pi$ (when scaling such that the curvature is $1$). Such a cone structure does not exist \cite[Corollary~3.5]{mondelloPanov}. In the language of cocycles, $\prod_1$ in Equation~\ref{eqn:gimbalMatrixAntipodal} is now a rotation by an angle strictly between $0$ and $\pi$, so $m_{\Gamma_1}$ is the product of two rotations about two different axes and thus only the identity when $\Theta_{e^\sim_l}=\Theta_{e^\sim_{1}}=2\pi$. In fact, this forces $\Theta_{e^\sim_1}=\dots=\Theta_{e^\sim_l}=2\pi$.  
\end{example}

\begin{example}[A non-hyperbolic manifold fulfilling all but six edge equations]
Consider the two tetrahedron triangulation of $S^3$ obtained by identifying the boundary of the first tetrahedron with the boundary of the second tetrahedron via the identity map. Assign the same length to all edges such that each tetrahedron has a hyperbolic structure individually but all of the six edges equations are violated. Perform several 1-4 moves on one or both of the tetrahedra preserving the hyperbolic structure on each tetrahedron. This gives an example of a non-hyperbolic manifold where more than $m-3o$ but not all edge equations are fulfilled. Compare this to the ideal case where Neumann-Zagier (see Appendix) state that all edge equations must be fulfilled whenever $m-o$ are fulfilled.
\end{example}

We give a characterization of gimbal lock in 1-vertex triangulations later in Theorem~\ref{thm:singleVertexTriangulationGimbalLockCondition}.

\section{Algorithm} \label{sec:algo}

\subsection{Overview}

Let $\trig$\negSpace{} be an oriented, finite 3-dimensional triangulation. Let $o=V(\trig)$ be the number of vertices and index the edges $E(\trig)=\{e_1,\dots,e_m\}$.
Assume we are given floating-point approximations for the edge parameters $\nu_{e_1},\dots,\nu_{e_m}$ such that the $\Theta_{e_1}, \dots, \Theta_{e_m}$ are approximately $2\pi$ (e.g, using the program Orb \cite{orb} based on the methods described in \cite[Section~2.2]{heardThesis}).

We will either fail in one of the steps or obtain intervals $\du{\nu_{e_1}},\dots,\du{\nu_{e_m}}$ guaranteed to contain values for $\nu_{e_1},\dots,\nu_{e_m}$ yielding a hyperbolic structure on $\trig$\negSpace{} as follows:
\begin{enumerate}
\renewcommand{\labelenumi}{\Roman{enumi}.}
\renewcommand{\theenumi}{\Roman{enumi}}
\item \label{step:submatrix} Find a subsystem of equations of full rank:\\
Since the given approximations for the edge parameters are close to a hyperbolic structure, evaluating the Jacobian
\begin{equation} \label{eqn:jacobian}
M=\left(\frac{\partial \Theta_{e_j}}{\partial \nu_{e_i}}\right)_{i=1,\dots,m;j=1,\dots,m}
\end{equation}
gives a matrix close to a singular matrix that we expect to have rank $m-3o$ (see Conjecture~\ref{conjecture:main}). Pick a suitable set of  $m-3o$ rows and columns such that the resulting submatrix of $M$ has full rank.\\
This corresponds to picking two partitions of the edges of $\trig$
\begin{eqnarray*}
E(\trig)=E^\sim \cup E^=&\mbox{with}& E^\sim=\{e^\sim_1,\dots,e^\sim_{3o}\},~E^= = \{e^=_1,\dots,e^=_{m-3o}\}\\
E(\trig)=E^\mathrm{fixed} \cup E^\mathrm{var}& \mbox{with}& E^\mathrm{fixed} = \{e^{\mathrm{fixed}}_1,\dots,e^{\mathrm{fixed}}_{3o}\}, ~E^\mathrm{var} = \{e^{\mathrm{var}}_1,\dots,e^{\mathrm{var}}_{m-3o}\}
\end{eqnarray*}
such that 
$$M'=\left(\frac{\partial \Theta_{e^=_j}}{\partial \nu_{e^\mathrm{var}_i}}\right)_{i=1,\dots,m-3o;j=1,\dots,m-3o}$$
has full rank near the given values. Keep the values of $\nu_{e^\mathrm{fixed}_i}$ fixed from now on.
\item \label{step:intervalNewton} Find an interval solution to the above subsystem of equations:\\
Using interval Newton method or Krawczyk test, find large enough intervals 
$$\du{\nu_{e^\mathrm{var}_1}},\dots,\du{\nu_{e^\mathrm{var}_{m-3o}}}$$
for the edges in $E^\mathrm{var}$ such that they contain a point where $\Theta_{e^=_j}=2\pi$ for every $e^=_j\in E^=$. To have intervals $\du{\nu_{e_1}},\dots,\du{\nu_{e_m}}$ for all edges of the triangulation, simply use the interval $[\nu_{e^\mathrm{fixed}_i},\nu_{e^\mathrm{fixed}_i}]$ containing only 
the fixed value $\nu_{e^\mathrm{fixed}_i}$ for each remaining edge $e^\mathrm{fixed}_i\in E^\mathrm{fixed}$.
\item \label{step:validGram} Ensure solutions are valid for a simplex:\\
Using the intervals for each edge, use interval arithmetic to verify the conditions of Lemma~\ref{lemma:singleGeomSimp} for each simplex $\Delta$.
\item \label{step:approxEdgeEqn} Ensure that the remaining edge equations are ``approximately'' fulfilled:\\
Use these intervals to compute intervals $\du{\Theta_{e^\sim_1}},\dots,\du{\Theta_{e^\sim_{3o}}}$ for the sums of dihedral angles adjacent to an edge in $E^\sim$ to ensure that $2\pi \in \du{\Theta_{e^\sim_j}}$ for every such edge.
\item \label{step:avoidGimbal} Ensure that the remaining edge equations are fulfilled exactly:\\
Find a gimbal loop $\myPath_1,\dots,\myPath_o$ for each vertex and consider the resulting gimbal function $g$. Ensure that $K=\du{\Theta_{e^\sim_1}}\times\cdots\times\du{\Theta_{e^\sim_{3o}}}$ avoids gimbal lock for the intervals computed in the previous step using interval arithmetics.
\end{enumerate}

\subsection{Computing the Jacobian}

Note that the formulas given in \cite[Lemma~2.5]{heardThesis} for the derivatives $\partial\theta_{ij}/\partial{v_{mn}}$ (where $v_{mn}=v_{nm}$) run into a division by zero when $\theta_{ij}=\pi/2$. To obtain formulas avoiding this problem, we take the total derivative of Equation~\ref{eqn:dihedral}
$$\frac{\partial\theta_{ij}}{\partial{v_{mn}}}
=\frac{-1}{\sqrt{c_{ii}c_{jj}-c_{ij}^2}} \left(\frac{\partial c_{ij}}{\partial v_{mn}} - \frac{c_{ij}}{2 c_{ii}}\cdot \frac{\partial c_{ii}}{\partial v_{mn}} - \frac{c_{ij}}{2c_{jj}}\cdot \frac{\partial c_{jj}}{\partial v_{mn}}\right)
 $$ and note that each $\partial c_{kl}/\partial v_{mn}=(-1)^{k+l} \partial (\det G_{kl}) / \partial v_{mn}$ is, up to sign, the sum of at most two cofactors of $G_{kl}$, namely, the ones corresponding to the $2\times 2$-matrices that can be obtained by deleting from $G_{kl}$ the row and column corresponding to row $m$ or $n$ and column $n$, respectively, $m$ of $G$.

Using this, we can compute the Jacobian $M$ in \eqref{eqn:jacobian} using floating point (in Step~\ref{step:submatrix}), respectively, interval arithmetics (in Step~\ref{step:intervalNewton}) avoiding the need for automatic differentiation.

\subsection{Step~\ref{step:submatrix}: Finding a submatrix of full rank}

This step is necessary since interval methods (e.g., interval Newton method or Krawczyk test) only work for systems with invertible Jacobian. To obtain a subsystem of full rank, we apply the following algorithm to $M$ with $h=m-3o$:
\begin{center}
\fbox{
\parbox{12.7cm}{
\begin{tabular}{rp{9.8cm}}
{\bf Input:} & Square-matrix $M=(m_{r,c})$ with expected rank $h$.\\
{\bf Output:} & Sets $R$ and $S$ of indices of $h$ rows, respectively, $h$ columns. The submatrix $M'$ of $M$ formed by these rows and columns will have full rank.\\
{\bf Algorithm:}
\end{tabular}
\begin{enumerate}
\itemsep0em 
\renewcommand{\labelenumi}{\arabic{enumi}.}
\renewcommand{\theenumi}{\arabic{enumi}.}
\renewcommand{\labelenumii}{\arabic{enumii}.}
\renewcommand{\theenumii}{\arabic{enumii}}
\item $R\leftarrow\{\}$. $C\leftarrow\{\}$.
\item Repeat $h$ times:\label{step:Repeat}
\begin{enumerate}
\itemsep0em 
\item Let $(r,c)$ be the index of the entry $m_{r,c}$ in $M$ with the largest absolute value when ignoring the rows in $R$ and columns in $C$.\label{substep:maxEntry}
\item Add multiples of row $r$ to all other rows of $M$ to make all entries except for $m_{r,c}$ in column $c$ zero.
\item Add multiples of column $c$ to all other columns of $M$ to make all entries except for $m_{r,c}$ in row $r$ zero.
\item $R\leftarrow R\cup\{r\}$. $C\leftarrow C\cup\{c\}$.
\end{enumerate}
\end{enumerate}
}
}
\end{center}
\begin{remark} Note that the algorithm has the following stability properties:
\begin{enumerate}
\item We obtain the same set of rows and columns of $M$ when permuting the rows or columns of the input matrix $M$, i.e., the result is obtained by applying the same permutation to $R$, respectively, $C$ --- unless there are ties in Step \ref{substep:maxEntry}.
\item Transposing $M$ results in interchanging $R$ and $C$.
\end{enumerate}
\end{remark}
\begin{remark}
This algorithm is a simplified version of $LDU$-factorization with full pivoting, i.e., a decomposition $M=PLDUP'$ where $P$ and $P'$ are permutation matrices, $L$ and $U$ unit-lower, respectively, unit-upper triangular matrices and $D$ a diagonal matrix.
 \end{remark}

\subsection{Step~\ref{step:intervalNewton}}

This step is a straightforward application of the interval Newton method or Krawczyk test to the equations $\Theta_{e^=_j}-2\pi = 0$ in variables $\nu_{e^\mathrm{var}_i}$ (keeping $\nu_{e^\mathrm{fixed}_i}$ fixed). Note that even for high precision solutions, computing the approximate inverse in the Krawczyk test using IEEE754 double-precision floating point numbers is usually sufficient.

If we are interested in increasing the precision of the solution, we can optionally perform the ordinary Newton method to the subsystem from Step~\ref{step:submatrix} before Step~\ref{step:intervalNewton}.

\subsection{Steps~\ref{step:validGram} and \ref{step:approxEdgeEqn}}

These conditions are straightforward to check with interval arithmetics given the comment about the Budan-Fourier theorem in Lemma~\ref{lemma:singleGeomSimp}.

\subsection{Step~\ref{step:avoidGimbal}: Finding gimbal loops}

We need to pick a gimbal loop $\myPath_k$ in each ``vertex link'' $(L_k,L^=_k)$. Note that this complex consists of the small hexagons (with alternating $\beta$- and $\gamma$-edges) of the doubly-truncated simplices (see Figure~\ref{fig:DoublyTruncated}) and polygons coming from the ends of the prisms (see Figure~\ref{fig:Prism} and \ref{fig:gimbalLoopsInTriangulation}). For a vertex $v_k$, pick a hexagon in $L^=_k$, mark it as ``used'' and starting with its boundary (oriented such that it traverses a $\gamma^{\sigma(0)\sigma(1)\sigma(2)}$ with $\sigma\in A_4$), expand this edge-loop until it touches each boundary component of $L^=_k$ as follows: pick a $\beta$-edge of the edge-loop that is adjacent to an unused hexagon $H$ and replace the $\beta$-edge by the five other edges of $H$, marking $H$ as ``used'' (see Figure~\ref{fig:findingAGimbalLoop}). It is faster to exhaust all $\beta$-edges of one hexagon first in breadth-first search manner (vs depth-first search) before moving on to the next hexagon.

\begin{figure}[h]
\begin{center}
\scalebox{0.14}{
\input{figures_gen/gimbalLoopAlgorithm.tex}}
\end{center}
\caption{Procedure to find a gimbal loop $\myPath$.} \label{fig:findingAGimbalLoop}
\end{figure}
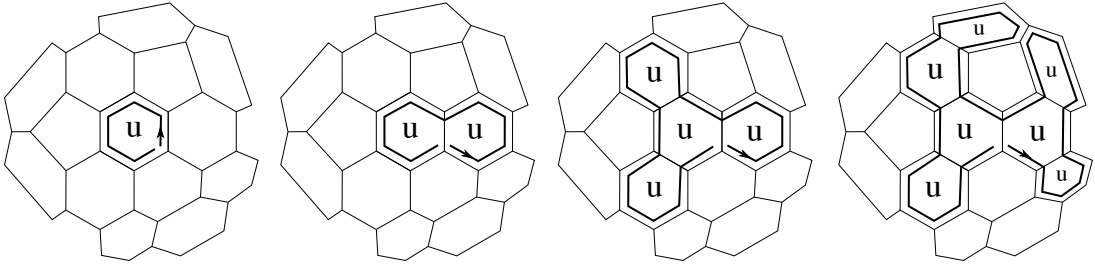

\subsection{Step~\ref{step:avoidGimbal}: Computing the gimbal function's derivative}

We then need to compute $[Dg(K)]$ which boils down to computing the derivatives $\partial m_{\myPath_k}/\partial T_i$ of the gimbal matrices $m_{\myPath_k}$ associated to the gimbal loops $\myPath_k$. Focusing on one $i$ and $k$, note that $m_{\myPath_k}$ is given by an alternating product the form
$$
\prod\nolimits_1 \cdot R_{T_i}\cdot \prod\nolimits_2  \cdot R_{T_i} \cdot \prod\nolimits_3 \cdot R_{T_i} \cdots \prod\nolimits_q
$$
where $\prod\nolimits_k$ stands for a product of $\beta$- and $\gamma$-matrices and rotations $R_{T_{i'}}$ with $i'\not=i$ ($q$ is actually at most three since an edge of $\trig$\negSpace{} has two ends which might or might not end in the same vertex). We can then compute the derivative as
\begin{eqnarray*}
\frac{\partial m_{\myPath_k}}{\partial T_i} &=& \phantom{+} \prod\nolimits_1 \cdot R'_{T_i}\cdot \prod\nolimits_2  \cdot R_{T_i} \cdot \prod\nolimits_3 \cdots R_{T_i} \cdot \prod\nolimits_q\\
 & & + ~ \cdots \\
 & & +\prod\nolimits_1 \cdot R_{T_i}\cdot \prod\nolimits_2  \cdot R_{T_i} \cdot \prod\nolimits_3 \cdots R'_{T_i} \cdot \prod\nolimits_q \quad \mbox{where}~
R'_\omega = \left(\begin{array}{ccc} -\sin\omega & -\cos\omega & \\ \phantom{-}\cos\omega & -\sin\omega & \\ & & 0\end{array}\right).
\end{eqnarray*}

We need to evaluate this using the intervals from Step~\ref{step:intervalNewton} for the $\beta$- and $\gamma$-matrices and the intervals $\du{\Theta_{e^\sim_1}},\dots,\du{\Theta_{e^\sim_{3o}}}$ from Step~\ref{step:approxEdgeEqn} for $R'_{T_i}$ and all the $R_{T_j}$.

\subsection{Step~\ref{step:avoidGimbal}: Verifying invertibility}

To show that a square matrix $m$ with real interval entries is invertible (in the sense used in Definition~\ref{def:gimbalLockAvoided} and \ref{def:trigGimbalFunction}), we can find an approximate inverse $n$ (usually IEEE754-double precision is sufficient) and verify that each entry of $mn-\mathrm{Id}$ has absolute value strictly less than $1/r^2$ where $r$ is the number of rows.

\section{Results}  \label{sec:results}

Our implementation of the algorithm in Section~\ref{sec:algo} is available at \cite{veriClosedRepo}. Lists of all knots and links from Theorem~\ref{thm:hyperbolicBranchedDoubleCovers} as well as the isomorphism signatures of the finite triangulations we used are available at \cite{veriClosedData}.

To produce the input to the algorithm, we used SnapPy \cite{SnapPy} to produce a finite triangulation of a manifold and Orb \cite{orb} to find unverified floating point edge parameters. Note that there are finite triangulations that admit hyperbolic structures but Orb is unable to find one. However, for all examples of hyperbolic manifolds considered here, we are always able to make Orb succeed in finding edge parameters by randomizing the finite triangulation in SnapPy several times.

We were able to verify a hyperbolic structure on a finite triangulation of each of the 11031 orientable closed manifolds in the Hodgson-Weeks census \cite{hwcensus} and the 21962 genus 2 surface bundles and 3100 genus 3 surface bundles in the census by Bell \cite{bellBundleCensus}. In particular, we have an independent proof of \cite[Theorem~5.2]{hikmot} that all manifolds in SnapPy's \texttt{OrientableClosedCensus} are hyperbolic.

To prove Theorem~\ref{thm:hyperbolicBranchedDoubleCovers}, we went through all knots and links up to 15, respectively, 14 crossings tabulated by Hoste-Thistlethwaite\footnote{See \texttt{HTLinkExteriors}, \texttt{AlternatingKnotExteriors}, and \texttt{NonalternatingKnotExteriors} in SnapPy.} and used SnapPy to produce the branched double cover. We either used the above method to prove that the resulting manifold is hyperbolic or used Regina \cite{Regina} to prove that it is not hyperbolic by
\begin{itemize}
\item finding the triangulation in Regina's census or
\item recognizing that the triangulation has a structure fitting one of Regina's\\ \texttt{StandardTriangulation}'s making it, e.g., a Seifert fiber space, or
\item finding an essential torus using normal surface theory. 
\end{itemize}
Note that the first two methods sometimes require some randomizations and simplifications of the triangulation in SnapPy to work and that the last method can be quite expensive.

The number of tetrahedra in the triangulations used to prove hyperbolicity was between 9 and 46. Most triangulations are single vertex, few have two vertices, and only one had three vertices.

Testing the algorithm on the bundle census and branched double covers was suggested by Nathan Dunfield since the geodesics isotopic to the components of the branching locus tend to be long and spun triangulations spinning about one of these geodesics very often fail to be geometric. By filling and drilling the triangulation, Dunfield found geometric spun triangulations (about a different geodesic) for the branched double covers of 42367 non-alternating knots and links up to 14 crossings, missing 1593.



\begin{remark} \label{remark:experimentEdgePartition}
We also investigate how the occurrence of gimbal lock depends on the edge partition $E^{\sim}\cup E^==E(\trig)$ with $|E^{\sim}|=3o$ where $o$ is number of vertices. For this, we fix a triangulation and a hyperbolic structure and check numerically whether the derivative $Dg$ of the gimbal function has singular values close to zero for different partitions (exhausting all partitions when $o\leq 2$ and sampling otherwise). We did this for several triangulations of orientable closed census manifolds including some with three and four vertices obtained by performing 1-4 moves. This lead to Conjecture~\ref{conjecture:main}.
\end{remark}

\section{Discussion} \label{sec:discussion}

Let $\trig$\negSpace{} be a finite, orientable triangulation with $o$ vertices and $m$ edges.
We have proven that $\trig$\negSpace{} admits a hyperbolic structure if the checks in each step of the algorithm in Section~\ref{sec:algo} pass. But are there hyperbolic structures which the algorithm cannot verify --- even as we increase the precision and give the algorithm better and better approximations of the edge lengths of the hyperbolic structure as input?

We conjecture that such hyperbolic structures are special and can always be avoided by a random perturbation. We state this in the following conjecture where ``generic'' means that a statement is true except for a closed measure zero set of hyperbolic structures on $\trig$ (we will see later that a natural measure exists on the space of all hyperbolic structures in Theorem~\ref{thm:hypStructAreSubmanifold}):

\begin{conjecture} \label{conjecture:main}
Let $\trig$\negSpace{} be a finite, orientable triangulation with $o$ vertices and $m$ edges admitting a hyperbolic structure. Then a generic hyperbolic structure on $\trig$\negSpace{} gives rise to edge lengths $l_{e_1},\dots, l_{e_m}>0$ or equivalently edge parameters $\nu_{e_1},\dots,\nu_{e_m}\leq -1$ such that 
\begin{enumerate}
\item $M$ in \eqref{eqn:jacobian} has rank $m-3o$ and \label{subconjecture:Rank}
\item any choice of $m-3o$ linearly independent rows from $M$ avoids gimbal lock. More precisely, for any partition $E(\trig)=E^\sim \cup E^=$ into $3o$ and $m-3o$ edges, we have that\label{subconjecture:gimbalLock}
\begin{enumerate}
\item[1.] the rows of $M$ corresponding to the edges in $E^=$ are linearly independent
\end{enumerate}
implies that
\begin{enumerate}
\item[2.] the derivative $Dg$ of the gimbal function in Definition~\ref{def:trigGimbalFunction} is invertible.
\end{enumerate}
\end{enumerate}
\end{conjecture}

Note that if Part~\eqref{subconjecture:gimbalLock} were false, the choice of the partition $E(\trig)=E^\sim\cup E^=$ made in Step~\ref{step:submatrix} could be such that Step~\ref{step:intervalNewton} passes but the algorithm fails later in Step~\ref{step:avoidGimbal}.

\begin{remark}
Numerically, we found that the converse of Part~\eqref{subconjecture:gimbalLock} is not true, i.e., we found examples where the gimbal function is invertible even though the chosen rows of $M$ are linearly dependent.
\end{remark}

\begin{remark}
Compare the conjecture to the ideal case where we do not know in general whether every cusped, finite volume hyperbolic 3-manifold has a geometric triangulation (see \cite{PetronioWeeks:partiallyFlatTrig,LuoSchleimerTillmann:virtualGeometricTrig,Goerner:DodecahedralGeometricTriangulation}).
\end{remark}

\subsection{The space of hyperbolic structures}

Let $\trig$\negSpace{} be a finite, orientable triangulation with $o$ vertices and $m$ edges admitting a hyperbolic structure. We are able to prove that the solution set of the edge equations has dimension $3o$. This is weaker than Part~\eqref{subconjecture:Rank} of Conjecture~\ref{conjecture:main} since it implies that the rank of $M$ is at most $m-3o$ (for example, $x^3=0$ yields a $0$-dimensional submanifold of the 1-dimensional $\R$, yet the Jacobian has rank $0$ at $0$ instead of the expected $1-0=1$).

Earlier, we defined a hyperbolic structure on $\trig$ as a compatible assignment of an isometry class of finite simplices in $\H^3$ to each simpex in $\trig$\negSpace{}. The space of hyperbolic structures on $\trig\negSpace{}$ can be described in the following ways:
\begin{enumerate}
\item $\edgeEqSol(\trig)\subset\R_{>0}^m$, the set of all tuples $(l_{e_1},\dots,l_{e_m})$ fulfilling the conditions of Theorem~\ref{thm:hypStruct} (since the edge lengths of a hyperbolic structure yield a point in $\edgeEqSol(\trig)$ and determine the hyperbolic structure uniquely).
\item The space of geodesic homeomorphisms $\trig\to\mathcal{M}$
in a fixed homotopy class in $[\trig\negSpace{},\mathcal{M}]$ where $\mathcal{M}$ is a hyperbolic 3-manifold homeomorphic to $\trig\negSpace{}$.
By Mostow rigidity, this yields all hyperbolic structures. By fixing the homotopy class, a hyperbolic structure gives a unique geodesic homeomorphism, instead of multiple related by the isometries of $\mathcal{M}$. \label{item:geodHomeo}
\item The space of all developing homeomorphisms, i.e., $\rho$-equivariant homeomorphisms $d:\tilde{\trig}\to\H^3$ where $p$ is a fixed vertex of $\trig\negSpace{}$ and $\rho:\pi_1(\trig\negSpace{},p)\to\myPSL{2}{\C}$ is a fixed geometric representation. This gives the homeomorphism $\trig\negSpace{}\to\mathcal{M}$ in \eqref{item:geodHomeo} when letting $\mathcal{M}=\H^3/\Gamma$ where $\Gamma=\mathrm{Im}(\rho)$.
\end{enumerate}

\begin{theorem} \label{thm:hypStructAreSubmanifold}
Let $\trig$\negSpace{} be a finite, orientable triangulation with $o$ vertices and $m$ edges admitting a hyperbolic structure.
The space $\edgeEqSol(\trig)$ of hyperbolic structures on $\trig$ is a $3o$-dimensional smooth submanifold of \,$\R_{>0}^m$, is (non-canonically) diffeomorphic to an open subset $U\subset \left(\H^3\right)^o$, and has a canonical measure induced from $\left(\H^3\right)^o$.
\end{theorem}

\begin{proof} We call a $\rho$-equivariant geodesic map $d:\tilde{\trig}\to\H^3$ a developing map.\\
{\bf Claim:} Fix a vertex $p$ of $\trig\negSpace{}$ and a geometric representation $\rho:\pi_1(\trig\negSpace{},p)\to\myPSL{2}{\C}$. Then developing maps are in 1-1 correspondence to $(\H^3)^o$.\\
Pick a lift $\tilde{v_1},\dots, \tilde{v_o}$ in $\tilde{\trig}\negSpace{}$ of each vertex of $v_1,\dots,v_o$ of $\trig\negSpace{}$. By $\rho$-equivariance, the images of all vertices of $\tilde{\trig}$ are determined by $(d(\tilde{v_1}), \dots, d(\tilde{v_o}))\in(\H^3)^o$. A geodesic map is determined uniquely by the image of all the vertices.\\
{\bf Claim:} Under the above assumptions, the homeomorphisms among the developing maps form an open subset $U\subset(\H^3)^o$. We have a bijection $l:U\to\edgeEqSol(\trig)$.\\
By invariance of domain, a developing map is a homeomorphism if and only if it is an embedding. A developing map is an embedding if and only if the image of each simplex is positively oriented. The condition on a single simplex in $\tilde{\trig}\negSpace$ to be positively oriented is open in the four vertex positions of the simplex. As a function of $(x_1,\dots,x_o)\in (\H^3)^o$, each of these four vertex positions is given by some $m x_k$ where $m$ is a hyperbolic isometry in $\Gamma=\mathrm{Im}(\rho)$. Thus, the condition is also open in $(\H^3)^o$. By $\rho$-equivariance, it is sufficient to check this for one lift $\tilde{\Delta}$ of each of the finitely many simplices $\Delta$ of $\trig\negSpace{}$, so $U$ is open.\\
{\bf Claim:} The map $L=i\circ l: U\to\R^m_{>0}$ is smooth where $i: \edgeEqSol(\trig)\hookrightarrow \R^m_{>0}$ is the inclusion.\\
The distance function $\H^3\times\H^3\to\R_{\geq 0}$ is smooth for all $(x,y)$ with $x\not= y$.\\
{\bf Claim:} The inverse $l^{-1}:\edgeEqSol(\trig)\to U$ is continuous.\\
Fix a simplex $\tilde{\Delta}$ in $\tilde{\trig}$ and $\sigma\in \permA_4$. Given a point in $\edgeEqSol(\trig)$, let us consider the developing map $d':\tilde{\trig}\to\H^3$ obtained by starting with $\tilde{\Delta}$ in $\sigma$-standard position (see Definition~\ref{def:standardPosSimplex}) and developing simplex by simplex.
Note that $d'$ is equivariant with respect to a conjugate but different representation $\rho':\pi_1(\trig\negSpace{}, p)\to\myPGL{2}{\C}$.
In other words, there is a unique element $h\in\myPGL{2}{\C}$ such that $\rho(\gamma)=h\circ \rho'(\gamma)\circ h^{-1}$ for all $\gamma\in\pi_1(\trig\negSpace{}, p)$ and $d=h\circ d'$ is $\rho$-equivariant.
We can compute $h$ as follows:
 fix $\gamma_1,\gamma_2,\gamma_3\in\pi_1(\trig\negSpace{},p)$ such that the attractive fixed points $p_i\in\partial \H^3$ of $\rho(\gamma_i)$ are distinct. Let $p'_i\in\partial\H^3$ be the attractive fixed points of $\rho'(\gamma_i)$. Then $h$ is the unique Moebius transformation with $p_i=h(p'_i)$. Note that the all $d'(\tilde{v_k})$ as well as all $\rho'(\gamma_i)$ and thus $h$ depend continuously on the point in $\edgeEqSol(\trig)$. Thus, 
 $(d(\tilde{v_1}), \dots, d(\tilde{v_o}))\in U\subset (\H^3)^o$ depends continuously on the point in $\edgeEqSol(\trig)$.\\
{\bf Claim:} The differential $DL$ of $L: U\to\R^m_{>0}$ is injective at every point in $U$.\\
Assume that $DL$ has non-trivial kernel. That is, there is a path $\gamma:(-1,1)\to U$ with $(d\gamma(t)/dt)_{|t=0}\not =0$ but $(d(L\circ\gamma)(t)/dt)_{|t=0}=0$. In other words, we have a 1-parameter family of developing maps $d:\tilde\trig\to \H^3$ where at least one of the points $d(\tilde{v_k})$ is moving with non-zero velocity but the derivatives of all the edge lengths is zero. As we shall see, this cannot happen.\\
From the edge lengths $(L\circ\gamma)(t)\in\edgeEqSol(\trig)$, we can construct a 1-parameter family of developing maps $d'$ as above. Since the derivatives of the lengths is zero, the vertices $d'(\tilde{v_k})$ all have velocity zero. It is not hard to see that the entries of the $h\in\myPGL{2}{\C}$ from above also have zero derivative. Thus, the points $d(\tilde{v_k})$ have zero velocity.\\
{\bf Claim:} The measure on $\edgeEqSol(\trig)$ induced from $\left(\H^3\right)^o$ is independent of the above choices.\\
A different choice of lifts $\tilde{v_1},\dots,\tilde{v_o}$ of vertices corresponds to the action of $(\H^3)^o$ by an element in $\Gamma^o$ and thus does not change the induced measure. A different choice of $\rho':\pi_1(\trig\negSpace{},p)\to\myPGL{2}{\C}$ is conjugate to $\rho$ by an element $h\in\myPGL{2}{\C}$. The measure on $(\H^3)^o$ is invariant under the action of $h$.
\end{proof}

\subsection{1-vertex triangulations} \label{sec:singleVertexTrigs} Consider a 1-vertex triangulation $\trig$ admitting a hyperbolic structure. Let $p$ be the vertex of $\trig$ and $\rho:\pi_1(\trig\negSpace{}, p)\to\myPSL{2}{\C}$ be a geometric representation. Note that each edge $e$ of $\trig$ forms a loop in $\pi_1(\trig\negSpace{}, p)$ and denote by $F_e=\{x\in \C P^1 : x = \rho(e)(x)\}$ the set of the corresponding fixed points.
\begin{theorem} \label{thm:singleVertexTriangulationGimbalLockCondition}
Let $E^\sim = \{e^\sim_1, e^\sim_2, e^\sim_3\}$ and $E^=$ be a partition of the edges of $\trig\negSpace{}$. Gimbal lock occurs for every hyperbolic structure on $\trig$ if there are $i\not= j$ with $F_{e^\sim_i}=F_{e^\sim_j}$. Otherwise, gimbal lock is generically avoided.
\end{theorem}

\begin{remark}
It is imaginable that there is a 1-vertex triangulation such that $|\{F_e : e\in E(\trig)\}|=2$ and gimbal lock occurs for each choice of $E^\sim$ giving a counterexample to Conjecture~\ref{conjecture:main}. 
\end{remark}

\begin{figure}[h]
\begin{center}
\scalebox{0.5}{
\input{figures_gen/geodesicTrajectory.tex}}
\end{center}
\caption{Edge in 1-vertex triangulation.\label{fig:edgeSingVertTrig}}
\end{figure}
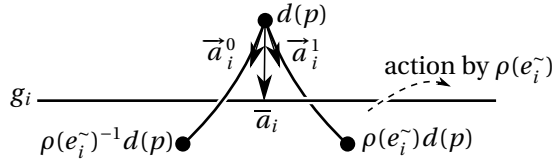

\begin{proof}
The edges $e^\sim_i$ become geodesics with a potential kink at the image of $p$ in $\mathcal{M}$. Let $\overrightarrow{a}_i^0$ and $\overrightarrow{a}_i^1$ be the directions (of unit length) in which the two ends of $e^\sim_i$ approach $d(p)\in\H^3$ in the developing embedding $d:\tilde{\trig}\to\H^3$. We can identify the tangent space at $d(p)$ with Euclidean 3-space such that the gimbal matrix $m_\Gamma$ is the product of six rotations (the order depending on the choice of gimbal loop $\Gamma$)
$$R_{T_1}^{\overrightarrow{a}_1^0},~R_{T_1}^{\overrightarrow{a}_1^1},~R_{T_2}^{\overrightarrow{a}_2^0},~R_{T_2}^{\overrightarrow{a}_2^1},~R_{T_3}^{\overrightarrow{a}_3^0},~\mbox{and}~R_{T_3}^{\overrightarrow{a}_3^1}$$
where $R_\omega^{\overrightarrow{v}}\in\mySO{3}$ denotes the rotation about $\overrightarrow{v}$ by angle $\omega$. Let $\overrightarrow{a}_{i,k}^j$ denote the $k$-th component of $\overrightarrow{a}_i^j$. At the point where all $T_i=2\pi$, we have
$$\frac{\partial m_\Gamma}{\partial T_i} = \frac{R_{T_i}^{\overrightarrow{a}_i^0}}{\partial T_i} + \frac{R_{T_i}^{\overrightarrow{a}_i^1}}{\partial T_i}=
\left(\begin{array}{ccc}
0 & -\overrightarrow{a}_{i,2}^0 & \overrightarrow{a}_{i,1}^0 \\
\overrightarrow{a}_{i,2}^0 & 0 & -\overrightarrow{a}_{i,0}^0 \\
-\overrightarrow{a}_{i,1}^0 & \overrightarrow{a}_{i,0}^0 & 0
\end{array}\right) +
\left(\begin{array}{ccc}
0 & -\overrightarrow{a}_{i,2}^1 & \overrightarrow{a}_{i,1}^1 \\
\overrightarrow{a}_{i,2}^1 & 0 & -\overrightarrow{a}_{i,0}^1 \\
-\overrightarrow{a}_{i,1}^1 & \overrightarrow{a}_{i,0}^1 & 0
\end{array}\right),
$$
$$\mbox{so}~\quad \frac{\partial g}{\partial T_i} = \left(-\big(\overrightarrow{a}_{i,2}^0+\overrightarrow{a}_{i,2}^1\big), \big(\overrightarrow{a}_{i,1}^0+\overrightarrow{a}_{i,1}^1\big), -\big(\overrightarrow{a}_{i,0}^0+\overrightarrow{a}_{i,0}^1\big)\right)=\left(-\overline{a}_{i,2}, \overline{a}_{i,1}, -\overline{a}_{i,0}\right)$$
where $\overline{a}_{i}=\overrightarrow{a}_i^0+\overrightarrow{a}_i^1.$ Thus, $Dg=\left(\partial g/\partial T_1, \partial g/\partial T_2, \partial g/\partial T_3\right)$ is invertible if $\overline{a}_1, \overline{a}_2,$ and $\overline{a}_3$ are linearly independent. Figure~\ref{fig:edgeSingVertTrig} shows that $\overline{a}_i$ is pointing from $d(p)$ to the geodesic $g_i$ spanned by $F_{e^\sim_i}$ (respectively pointing to $F_{e^\sim_i}$ if $|F_{e^\sim_i}|=1$) and $\overline{a}_i=0$ if it lies on that geodesic.\\
{\bf Claim:} Given a point $x\in\H^3$, consider the directions $\overline{a}_i$ from $x$ to the point on $g_i$ closest to $x$ (respectively pointing to $F_{e^\sim_i}$ if $|F_{e^\sim_i}|=1$). The set $V$ of $x$ where these directions lie in a plane is closed and has measure zero if all sets $F_{e^\sim_i}$ are distinct. Otherwise $V=\H^3$.\\
Let $f:\H^3\setminus \cup_i g_i\to \R, x\mapsto \omega(\overline{a}_1,\overline{a}_2,\overline{a}_3)$ where $\omega$ is the volume form on $\H^3$. Since $\overline{a}_i$ is the Hodge dual to the gradient of the distance of $x$ to $g_i$, $f$ is analytic. It is not hard to see that there is some $x$ with $f(x)\not = 0$ as long as no two $F_{e^\sim_i}$ and $F_{e^\sim_j}$ are equal. Thus $f^{-1}(0)$ and $V=f^{-1}(0)\cup \cup_i g_i$ have measure zero.
\end{proof}

\section{Appendix: The uncited theorem HIKMOT relies on} \label{sec:hikmotGap}

As pointed out in the introduction, the full system of equations for finding a hyperbolic structure on a triangulation fails to have rank equal to the number of variables in both the finite and ideal case. Thus, verification by interval methods in the ideal case also relies on a theorem showing that if a suitable subsystem is fulfilled, all equations are fulfilled. Unfortunately, the paper \cite{hikmot} did not cite this crucial theorem correctly. We want to point out that the algorithm in \cite{hikmot} is correct and its implementation can be trusted since the theorem it relies on has been proven by \cite[Lemma~2.4]{moser}. However the theorem from \cite{NeumannZagier} cited in \cite{hikmot} is insufficient since it assumes hyperbolicity to begin with and its conclusion is too weak. Being the counterpart of the result in this paper in the ideal case, we will briefly state the correct theorem missing from \cite{hikmot}.

Recall, that given an ideal triangulation with $o$ vertices, Thurston \cite[Chapter 4]{ThurstonNotes} gave compatibility equations in one complex variable $z_1,\dots,z_m$ per tetrahedron which are all of the form
$$\sum_{j=1}^m a_{r,j} \log\Big(z_j\Big) + b_{r,j} \log\Big(\frac{1}{1-z_j}\Big) + c_{r,j}\log\Big(1-\frac{1}{z_j}\Big) - 2\pi \myI d_r = 0$$
such that a solution with $\myIm(z_j)>0$ (called a geometric solution) yields a complete hyperbolic structure on the manifold obtained by filling in some cusps. There is one such equation with $d_r=1$ for each of the $m$ edges, one with $d_r=1$ for each filled cusped and one (sometimes two are given, but it is easy to see that one is sufficient) with $d_r=0$ for each unfilled cusp. Note that this system of equations consists of $m+o$ equations in $m$ variables and thus is overdetermined. To find a non-overdetermined system, let $\alpha_{r,j}=a_{r,j}-c_{r,j}$ and $\beta_{r,j}=-b_{r,j}+c_{r,j}$ and consider the $m\times m$-matrices $A=(\alpha_{r,j})$ and $B=(\beta_{r,j})$ where $r$ ranges over the edge equations:
\begin{theorem}
The matrix $(A|B)$ has rank $m-o$. Pick $m-o$ linearly independent rows of $(A|B)$ and consider the system of equations consisting of the corresponding $m-o$ edge equations and the $o$ cusp equations.
\begin{enumerate}
\item Any solution to this system fulfills the remaining $o$ edge equations as well.\label{item:moser1}
\item Near a geometric solution, the Jacobian of this system of equations is invertible.\label{item:moser2}
\end{enumerate}
\end{theorem}
Note that \eqref{item:moser1} is sufficient for an algorithm to prove a manifold to be hyperbolic. Statement \eqref{item:moser2} ensures that such an algorithm succeeds if given a solution close enough to the geometric one.

\cite{NeumannZagier} states the rank of $(A|B)$ but assumes hyperbolicity. Neumann revisted the result in \cite {NeumannComb} to give a purely combinatorial statement where the rank of $(A|B)$ occurs as rank of the map $\beta$ in a certain chain complex. However, even Theorem~4.1 in \cite{NeumannComb} only implies that the remaining $o$ edge equations in the above theorem are fulfilled modulo $2\pi\myI \Q$ since it does not involve the $d_r$ of the edge equations. For a proof of the above Theorem, see \cite[Section~2.3.1]{moser}.

\bibliographystyle{hamsalphaMatthias}
\bibliography{verifyingClosedManifolds}

\end{document}

%% file: figures_gen/anglesInTet.tex
\begingroup
 \setlength{\unitlength}{0.8pt}
 \begin{picture}(179.01929,147.44289)
 \put(0,0){\includegraphics{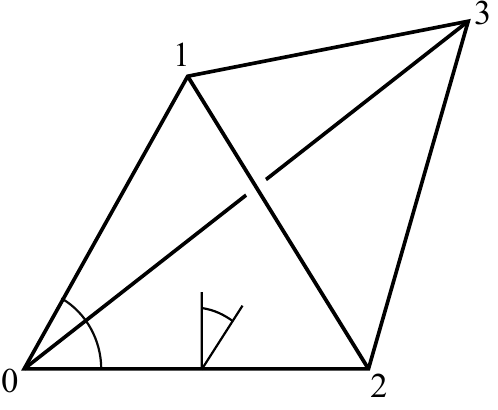}}

\definecolor{inkcol1}{rgb}{0.0,0.0,0.0}
   \put(34.332917,30.710579){\rotatebox{360.0}{\makebox(0,0)[tl]{\strut{}{
    \begin{minipage}[h]{123.948104pt}
\textcolor{inkcol1}{\normalsize{$\eta_{0,12}$}}\\
\end{minipage}}}}}%

\definecolor{inkcol1}{rgb}{0.0,0.0,0.0}
   \put(73.8678,53.036179){\rotatebox{360.0}{\makebox(0,0)[tl]{\strut{}{
    \begin{minipage}[h]{123.948104pt}
\textcolor{inkcol1}{\normalsize{$\theta_{13}$}}\\
\end{minipage}}}}}%

\definecolor{inkcol1}{rgb}{0.0,0.0,0.0}
   \put(14.798034,84.198969){\rotatebox{360.0}{\makebox(0,0)[tl]{\strut{}{
    \begin{minipage}[h]{123.948104pt}
\textcolor{inkcol1}{\normalsize{$l_{01}$}}\\
\end{minipage}}}}}%

 \end{picture}
\endgroup

%% file: figures_gen/doublyTruncatedCocycle.tex
\begingroup
 \setlength{\unitlength}{0.8pt}
 \begin{picture}(231.02043,183.89163)
 \put(0,0){\includegraphics{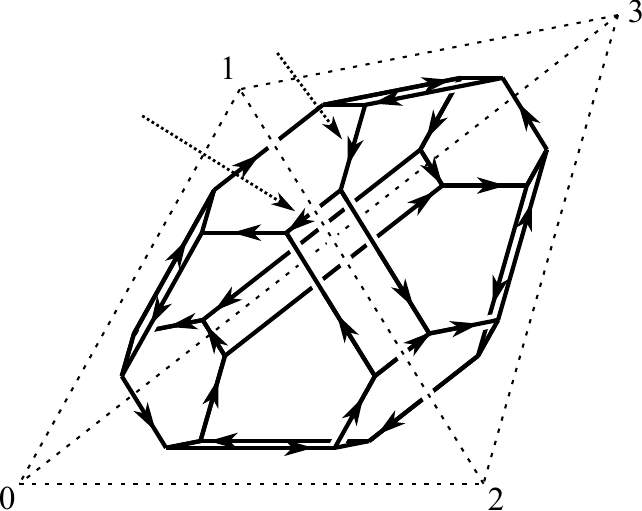}}

\definecolor{inkcol1}{rgb}{0.0,0.0,0.0}
   \put(88.33367,179.337569){\rotatebox{360.0}{\makebox(0,0)[tl]{\strut{}{
    \begin{minipage}[h]{123.948104pt}
\textcolor{inkcol1}{\normalsize{$\beta^{132}$}}\\
\end{minipage}}}}}%

\definecolor{inkcol1}{rgb}{0.0,0.0,0.0}
   \put(142.93021,93.552469){\rotatebox{360.0}{\makebox(0,0)[tl]{\strut{}{
    \begin{minipage}[h]{123.948104pt}
\textcolor{inkcol1}{\normalsize{$\alpha^{123}$}}\\
\end{minipage}}}}}%

\definecolor{inkcol1}{rgb}{0.0,0.0,0.0}
   \put(38.250965,156.123629){\rotatebox{360.0}{\makebox(0,0)[tl]{\strut{}{
    \begin{minipage}[h]{123.948104pt}
\textcolor{inkcol1}{\normalsize{$\gamma^{123}$}}\\
\end{minipage}}}}}%

 \end{picture}
\endgroup

%% file: figures_gen/doublyTruncatedCocycle_Prism.tex
\begingroup
 \setlength{\unitlength}{0.8pt}
 \begin{picture}(69.039253,183.89163)
 \put(0,0){\includegraphics{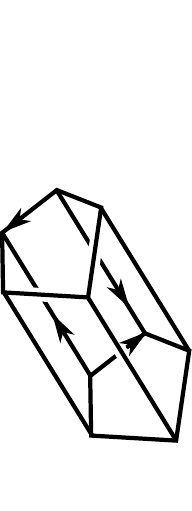}}
 \end{picture}
\endgroup

%% file: figures_gen/tetInStandardPosition.tex
\begingroup
 \setlength{\unitlength}{0.8pt}
 \begin{picture}(520.52539,267.66199)
 \put(0,0){\includegraphics{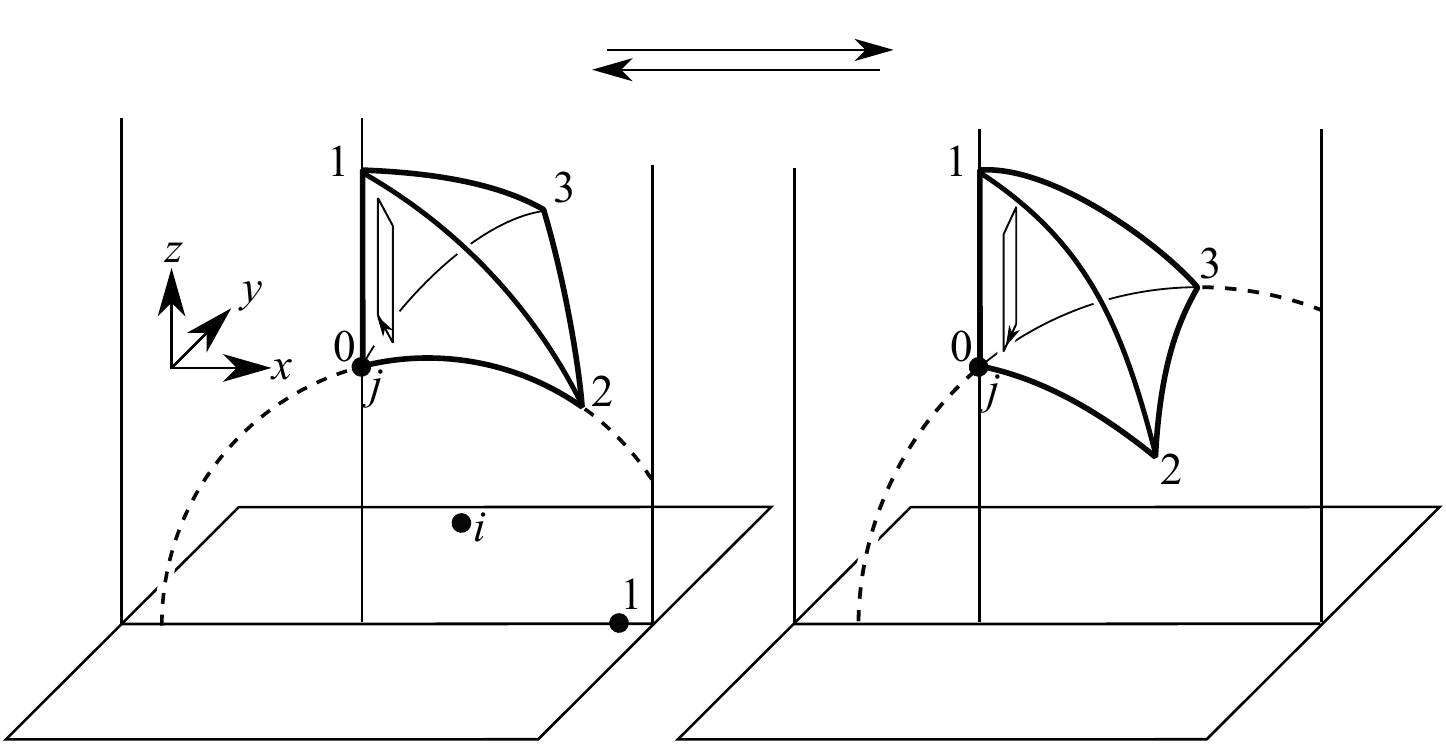}}

\definecolor{inkcol1}{rgb}{0.0,0.0,0.0}
   \put(46.257531,223.429199){\rotatebox{360.0}{\makebox(0,0)[tl]{\strut{}{
    \begin{minipage}[h]{123.948104pt}
\textcolor{inkcol1}{\normalsize{$\mathbb{H}^2$}}\\
\end{minipage}}}}}%

\definecolor{inkcol1}{rgb}{0.0,0.0,0.0}
   \put(2.7507959,79.302489){\rotatebox{360.0}{\makebox(0,0)[tl]{\strut{}{
    \begin{minipage}[h]{123.948104pt}
\textcolor{inkcol1}{\normalsize{$\mathbb{H}^3$}}\\
\end{minipage}}}}}%

\definecolor{inkcol1}{rgb}{0.0,0.0,0.0}
   \put(142.44669,154.428159){\rotatebox{360.0}{\makebox(0,0)[tl]{\strut{}{
    \begin{minipage}[h]{123.948104pt}
\textcolor{inkcol1}{\normalsize{$\gamma_{012}$}}\\
\end{minipage}}}}}%

\definecolor{inkcol1}{rgb}{0.0,0.0,0.0}
   \put(367.46449,147.073319){\rotatebox{360.0}{\makebox(0,0)[tl]{\strut{}{
    \begin{minipage}[h]{123.948104pt}
\textcolor{inkcol1}{\normalsize{$\gamma_{013}$}}\\
\end{minipage}}}}}%

\definecolor{inkcol1}{rgb}{0.0,0.0,0.0}
   \put(225.11595,268.742499){\rotatebox{360.0}{\makebox(0,0)[tl]{\strut{}{
    \begin{minipage}[h]{123.948104pt}
\textcolor{inkcol1}{\normalsize{action by $\gamma_{012}$}}\\
\end{minipage}}}}}%

\definecolor{inkcol1}{rgb}{0.0,0.0,0.0}
   \put(223.60464,238.680859){\rotatebox{360.0}{\makebox(0,0)[tl]{\strut{}{
    \begin{minipage}[h]{123.948104pt}
\textcolor{inkcol1}{\normalsize{action by $\gamma_{013}$}}\\
\end{minipage}}}}}%

\definecolor{inkcol1}{rgb}{0.0,0.0,0.0}
   \put(11.558057,252.880229){\rotatebox{360.0}{\makebox(0,0)[tl]{\strut{}{
    \begin{minipage}[h]{156.762136pt}
\textcolor{inkcol1}{\normalsize{$\Delta$ in (0,1,2,3)-standard position}}\\
\end{minipage}}}}}%

\definecolor{inkcol1}{rgb}{0.0,0.0,0.0}
   \put(326.69842,251.594099){\rotatebox{360.0}{\makebox(0,0)[tl]{\strut{}{
    \begin{minipage}[h]{169.451928pt}
\textcolor{inkcol1}{\normalsize{$\Delta$ in (0,1,3,2)-standard position}}\\
\end{minipage}}}}}%

 \end{picture}
\endgroup

%% file: figures_gen/gimbalLoop.tex
\begingroup
 \setlength{\unitlength}{0.8pt}
 \begin{picture}(864.069244299,877.776075153)
 \put(0,0){\includegraphics{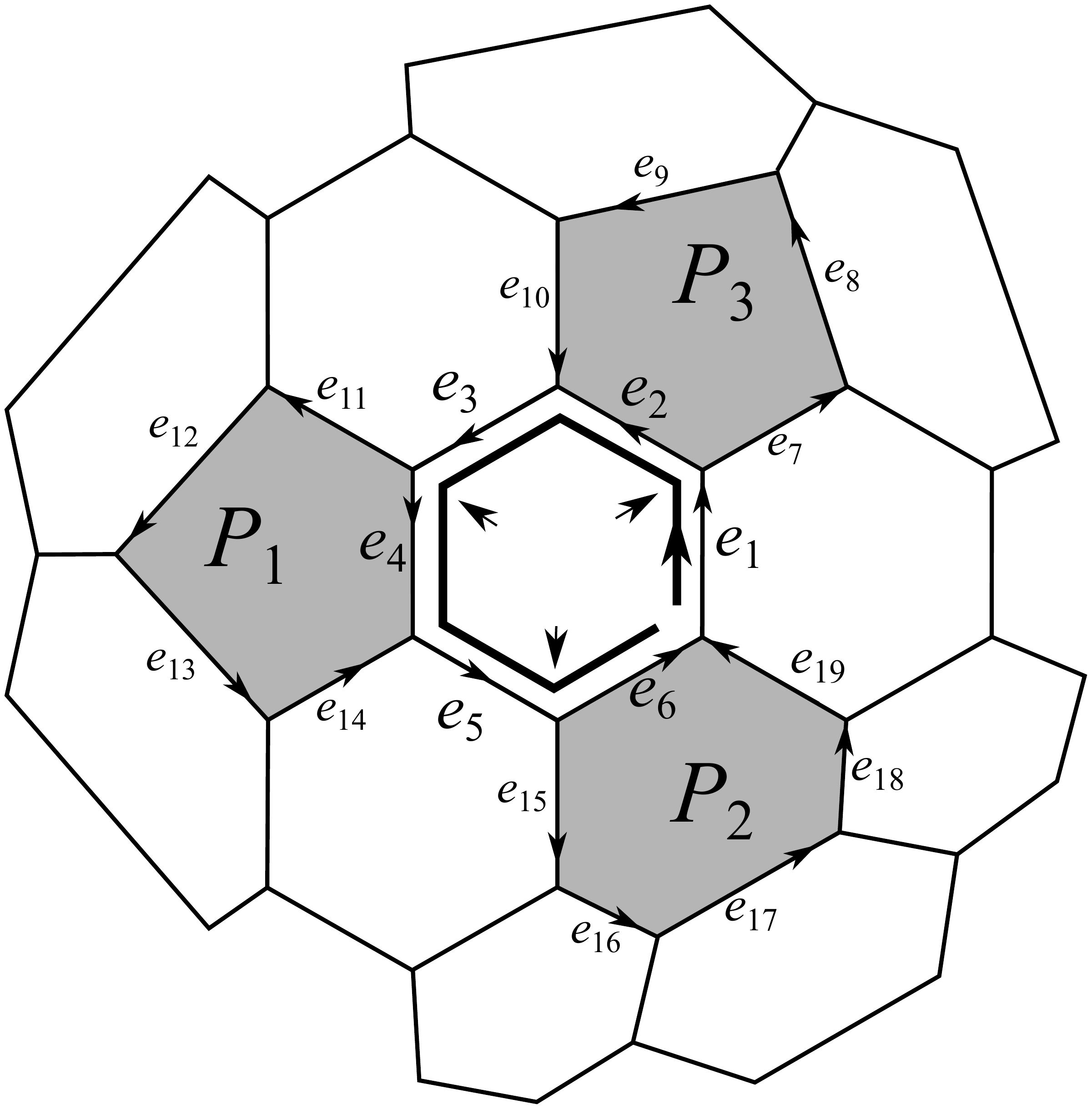}}

\definecolor{inkcol1}{rgb}{0.0,0.0,0.0}
   \put(467.97351,437.371384153){\rotatebox{360.0}{\makebox(0,0)[tl]{\strut{}{
    \begin{minipage}[h]{653.217728pt}
\end{minipage}}}}}%

 \end{picture}
\endgroup

%% file: figures_gen/gimbalLoopExtended.tex
\begingroup
 \setlength{\unitlength}{0.8pt}
 \begin{picture}(864.069244299,877.776075153)
 \put(0,0){\includegraphics{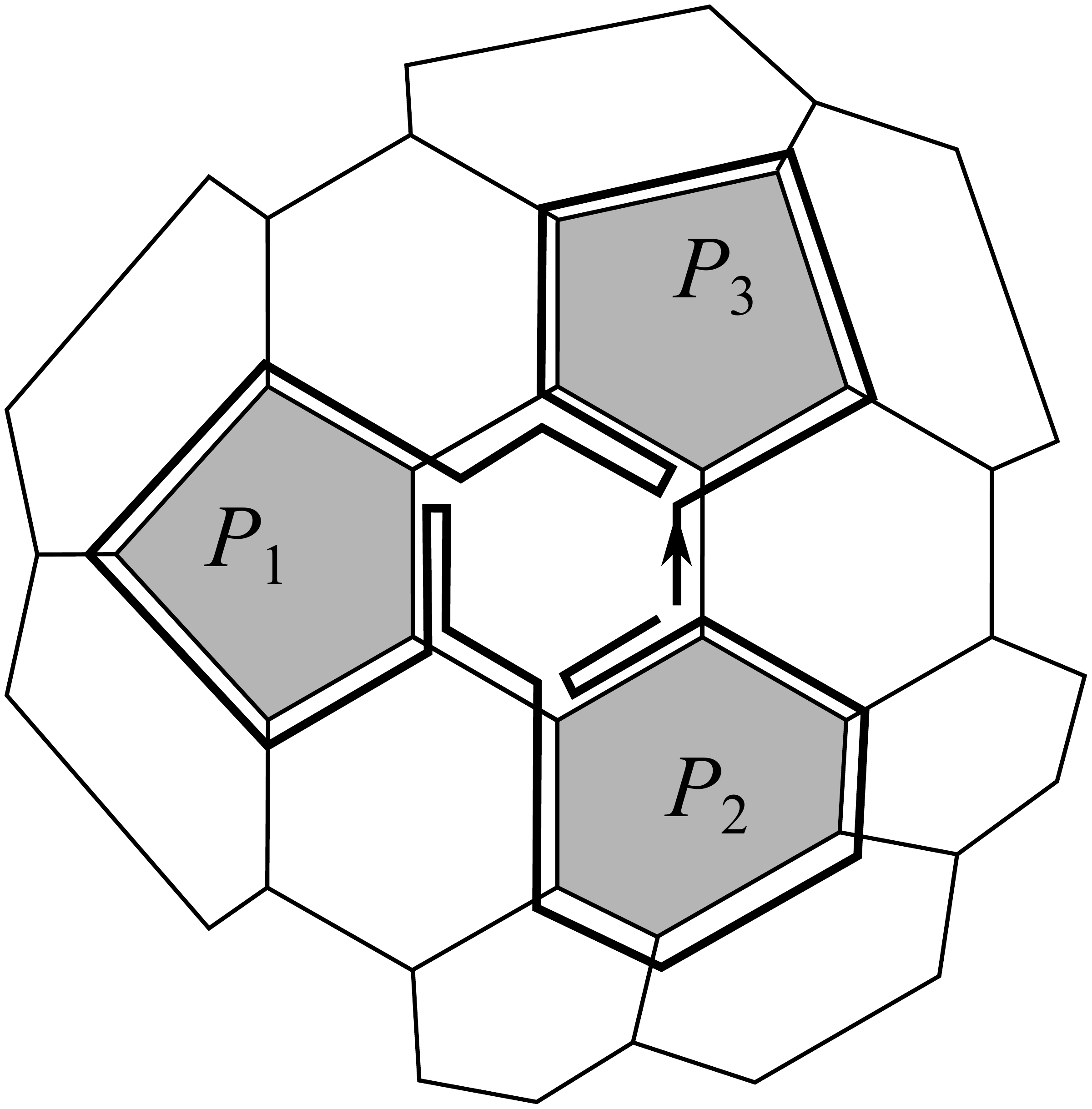}}
 \end{picture}
\endgroup

%% file: figures_gen/gimbalLoopSkeleton.tex
\begingroup
 \setlength{\unitlength}{0.8pt}
 \begin{picture}(864.069244299,877.776075153)
 \put(0,0){\includegraphics{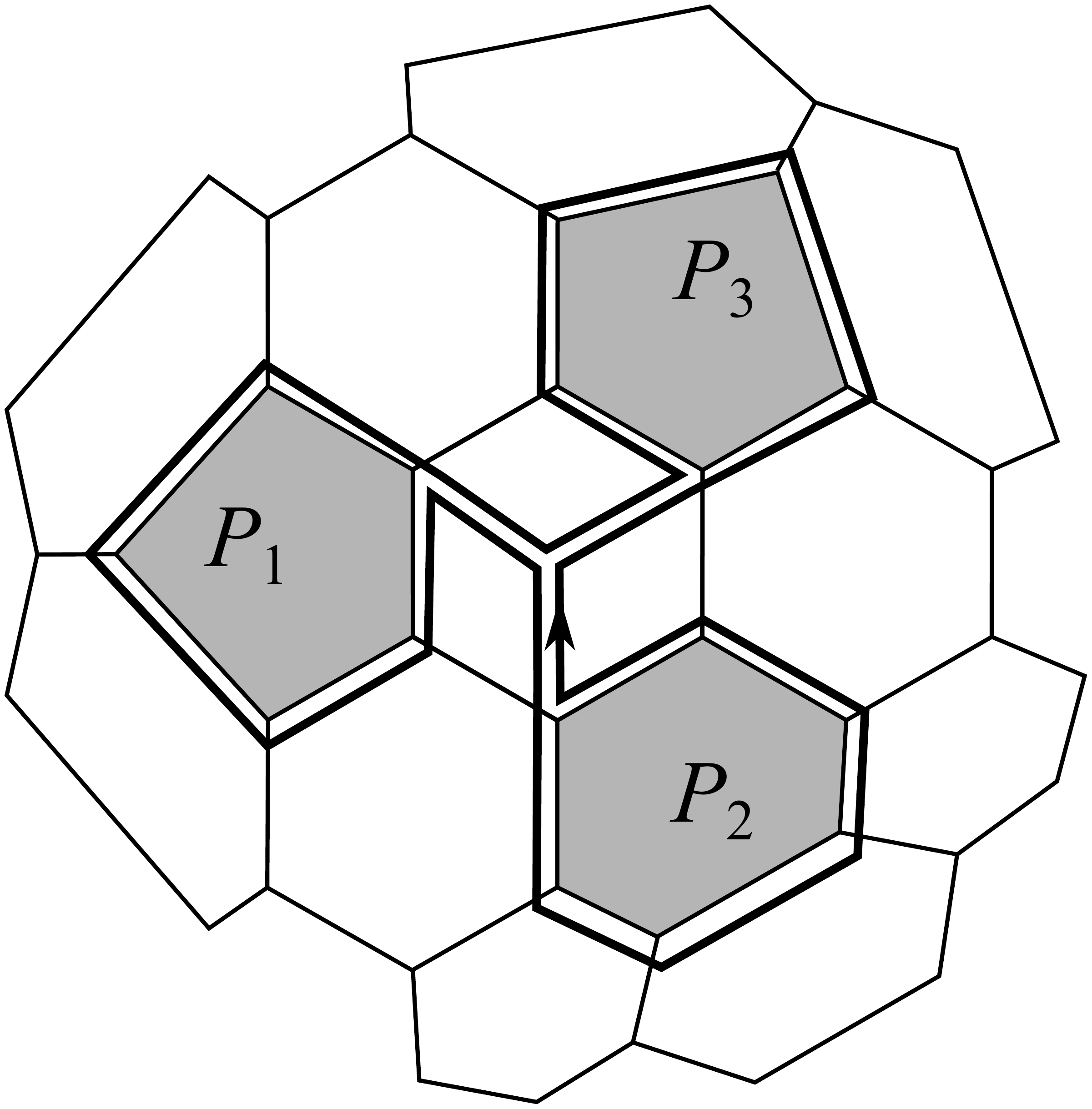}}
 \end{picture}
\endgroup

%% file: figures_gen/gimbalLoopTriangulation.tex
\begingroup
 \setlength{\unitlength}{0.8pt}
 \begin{picture}(542.121285834,261.419153539)
 \put(0,0){\includegraphics{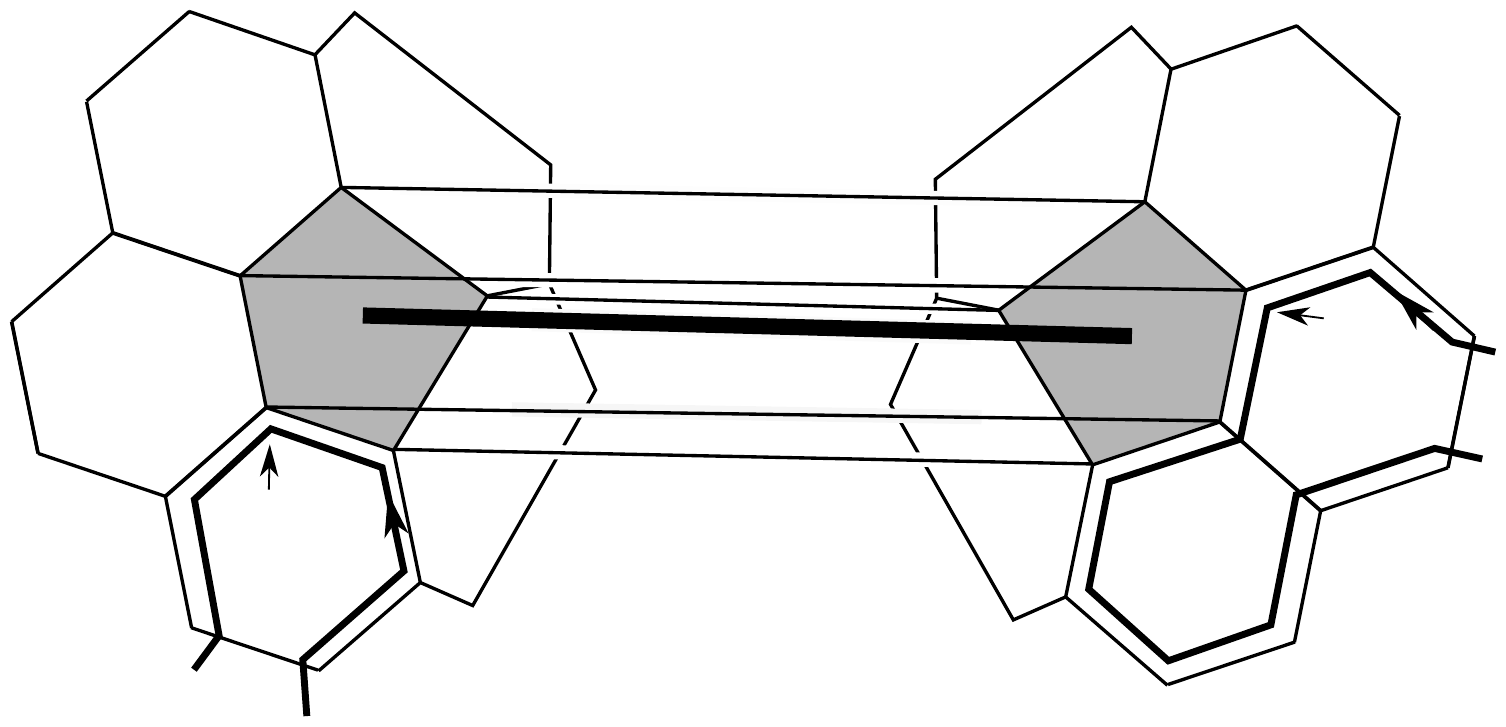}}

\definecolor{inkcol1}{rgb}{0.0,0.0,0.0}
   \put(242.20801,137.349136539){\rotatebox{360.0}{\makebox(0,0)[tl]{\strut{}{
    \begin{minipage}[h]{223.894728pt}
\textcolor{inkcol1}{\large{$e^\sim_j$}}\\
\end{minipage}}}}}%

\definecolor{inkcol1}{rgb}{0.0,0.0,0.0}
   \put(100.427717,93.0427865386){\rotatebox{360.0}{\makebox(0,0)[tl]{\strut{}{
    \begin{minipage}[h]{223.894728pt}
\textcolor{inkcol1}{\large{$\gamma^{012}_{\Delta_i}$}}\\
\end{minipage}}}}}%

\definecolor{inkcol1}{rgb}{0.0,0.0,0.0}
   \put(49.475426,123.318796539){\rotatebox{360.0}{\makebox(0,0)[tl]{\strut{}{
    \begin{minipage}[h]{223.894728pt}
\textcolor{inkcol1}{\large{$\beta^{013}_{\Delta_i}$}}\\
\end{minipage}}}}}%

\definecolor{inkcol1}{rgb}{0.0,0.0,0.0}
   \put(148.42625,87.1352765386){\rotatebox{360.0}{\makebox(0,0)[tl]{\strut{}{
    \begin{minipage}[h]{223.894728pt}
\textcolor{inkcol1}{\large{$\beta^{021}_{\Delta_i}$}}\\
\end{minipage}}}}}%

\definecolor{inkcol1}{rgb}{0.0,0.0,0.0}
   \put(130.70371,38.3982985386){\rotatebox{360.0}{\makebox(0,0)[tl]{\strut{}{
    \begin{minipage}[h]{223.894728pt}
\textcolor{inkcol1}{\large{$\gamma^{023}_{\Delta_i}$}}\\
\end{minipage}}}}}%

\definecolor{inkcol1}{rgb}{0.0,0.0,0.0}
   \put(31.752882,68.6742995386){\rotatebox{360.0}{\makebox(0,0)[tl]{\strut{}{
    \begin{minipage}[h]{223.894728pt}
\textcolor{inkcol1}{\large{$\gamma^{031}_{\Delta_i}$}}\\
\end{minipage}}}}}%

\definecolor{inkcol1}{rgb}{0.0,0.0,0.0}
   \put(245.9002,93.7812165386){\rotatebox{360.0}{\makebox(0,0)[tl]{\strut{}{
    \begin{minipage}[h]{223.894728pt}
\textcolor{inkcol1}{\large{$\alpha^{012}_{\Delta_i}$}}\\
\end{minipage}}}}}%

\definecolor{inkcol1}{rgb}{0.0,0.0,0.0}
   \put(414.26431,90.0890265386){\rotatebox{360.0}{\makebox(0,0)[tl]{\strut{}{
    \begin{minipage}[h]{223.894728pt}
\textcolor{inkcol1}{\large{$\gamma^{103}_{\Delta_i}$}}\\
\end{minipage}}}}}%

\definecolor{inkcol1}{rgb}{0.0,0.0,0.0}
   \put(355.18919,79.7508925386){\rotatebox{360.0}{\makebox(0,0)[tl]{\strut{}{
    \begin{minipage}[h]{223.894728pt}
\textcolor{inkcol1}{\large{$\beta^{102}_{\Delta_i}$}}\\
\end{minipage}}}}}%

\definecolor{inkcol1}{rgb}{0.0,0.0,0.0}
   \put(371.43485,25.1063905386){\rotatebox{360.0}{\makebox(0,0)[tl]{\strut{}{
    \begin{minipage}[h]{223.894728pt}
\textcolor{inkcol1}{\large{$\gamma^{120}_{\Delta_i}$}}\\
\end{minipage}}}}}%

\definecolor{inkcol1}{rgb}{0.0,0.0,0.0}
   \put(440.84812,22.8910865386){\rotatebox{360.0}{\makebox(0,0)[tl]{\strut{}{
    \begin{minipage}[h]{223.894728pt}
\textcolor{inkcol1}{\large{$\beta^{123}_{\Delta_i}$}}\\
\end{minipage}}}}}%

\definecolor{inkcol1}{rgb}{0.0,0.0,0.0}
   \put(473.33944,62.7667815386){\rotatebox{360.0}{\makebox(0,0)[tl]{\strut{}{
    \begin{minipage}[h]{223.894728pt}
\textcolor{inkcol1}{\large{$\gamma^{132}_{\Delta_i}$}}\\
\end{minipage}}}}}%

\definecolor{inkcol1}{rgb}{0.0,0.0,0.0}
   \put(513.21514,174.271076539){\rotatebox{360.0}{\makebox(0,0)[tl]{\strut{}{
    \begin{minipage}[h]{223.894728pt}
\textcolor{inkcol1}{\large{$\gamma^{302}_{\Delta_{i'}}$}}\\
\end{minipage}}}}}%

\definecolor{inkcol1}{rgb}{0.0,0.0,0.0}
   \put(453.40158,192.732066539){\rotatebox{360.0}{\makebox(0,0)[tl]{\strut{}{
    \begin{minipage}[h]{223.894728pt}
\textcolor{inkcol1}{\large{$\beta^{301}_{\Delta_{i'}}$}}\\
\end{minipage}}}}}%

\definecolor{inkcol1}{rgb}{0.0,0.0,0.0}
   \put(454.96789,137.425266539){\rotatebox{360.0}{\makebox(0,0)[tl]{\strut{}{
    \begin{minipage}[h]{223.894728pt}
\textcolor{inkcol1}{\large{$\gamma^{310}_{\Delta_{i'}}$}}\\
\end{minipage}}}}}%

\definecolor{inkcol1}{rgb}{0.0,0.0,0.0}
   \put(481.70397,122.235906539){\rotatebox{360.0}{\makebox(0,0)[tl]{\strut{}{
    \begin{minipage}[h]{223.894728pt}
\textcolor{inkcol1}{\large{$\gamma^{321}_{\Delta_{i'}}$}}\\
\end{minipage}}}}}%

 \end{picture}
\endgroup

%% file: figures_gen/gimbalLoopAlgorithm.tex
\begingroup
 \setlength{\unitlength}{0.8pt}
 \begin{picture}(3671.19718515,877.776075153)
 \put(0,0){\includegraphics{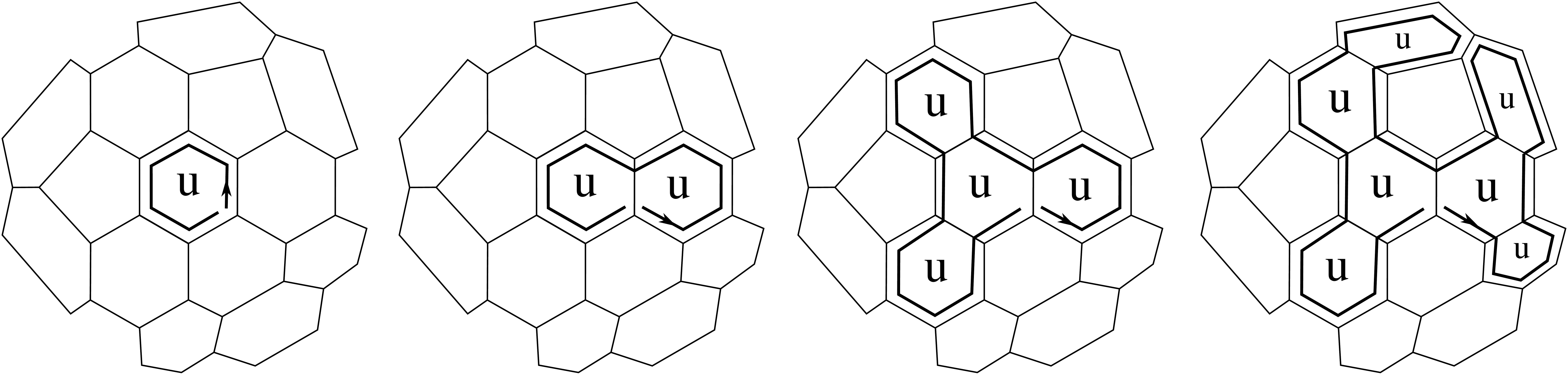}}
 \end{picture}
\endgroup

%% file: figures_gen/geodesicTrajectory.tex
\begingroup
 \setlength{\unitlength}{0.8pt}
 \begin{picture}(481.917925926,135.384351672)
 \put(0,0){\includegraphics{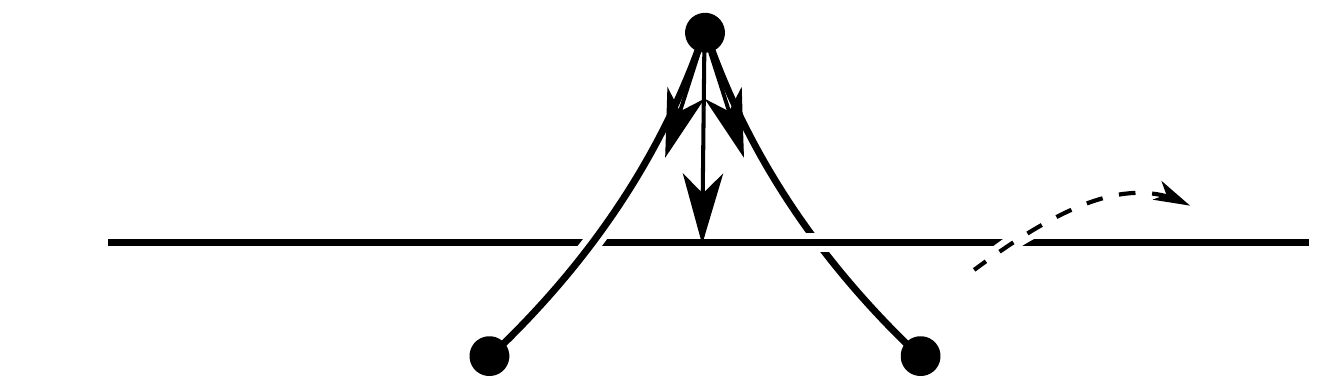}}

\definecolor{inkcol1}{rgb}{0.0,0.0,0.0}
   \put(363.998918,90.8438821317){\rotatebox{360.0}{\makebox(0,0)[tl]{\strut{}{
    \begin{minipage}[h]{318.389136pt}
\textcolor{inkcol1}{\LARGE{action by $\rho(e^\sim_i)$}}\\
\end{minipage}}}}}%

\definecolor{inkcol1}{rgb}{0.0,0.0,0.0}
   \put(261.641978,140.088202132){\rotatebox{360.0}{\makebox(0,0)[tl]{\strut{}{
    \begin{minipage}[h]{318.389136pt}
\textcolor{inkcol1}{\LARGE{$d(p)$}}\\
\end{minipage}}}}}%

\definecolor{inkcol1}{rgb}{0.0,0.0,0.0}
   \put(342.446788,22.9398021317){\rotatebox{360.0}{\makebox(0,0)[tl]{\strut{}{
    \begin{minipage}[h]{318.389136pt}
\textcolor{inkcol1}{\LARGE{$\rho(e^\sim_i)d(p)$}}\\
\end{minipage}}}}}%

\definecolor{inkcol1}{rgb}{0.0,0.0,0.0}
   \put(38.144725,22.9398021317){\rotatebox{360.0}{\makebox(0,0)[tl]{\strut{}{
    \begin{minipage}[h]{318.389136pt}
\textcolor{inkcol1}{\LARGE{$\rho(e^\sim_i)^{-1}d(p)$}}\\
\end{minipage}}}}}%

\definecolor{inkcol1}{rgb}{0.0,0.0,0.0}
   \put(241.735558,40.7561821317){\rotatebox{360.0}{\makebox(0,0)[tl]{\strut{}{
    \begin{minipage}[h]{318.389136pt}
\textcolor{inkcol1}{\LARGE{$\overline{a}_i$}}\\
\end{minipage}}}}}%

\definecolor{inkcol1}{rgb}{0.0,0.0,0.0}
   \put(271.524868,107.784292132){\rotatebox{360.0}{\makebox(0,0)[tl]{\strut{}{
    \begin{minipage}[h]{318.389136pt}
\textcolor{inkcol1}{\LARGE{$\overrightarrow{a}_i^1$}}\\
\end{minipage}}}}}%

\definecolor{inkcol1}{rgb}{0.0,0.0,0.0}
   \put(191.583003,108.155532132){\rotatebox{360.0}{\makebox(0,0)[tl]{\strut{}{
    \begin{minipage}[h]{318.389136pt}
\textcolor{inkcol1}{\LARGE{$\overrightarrow{a}_i^0$}}\\
\end{minipage}}}}}%

\definecolor{inkcol1}{rgb}{0.0,0.0,0.0}
   \put(9.740918,59.8911221317){\rotatebox{360.0}{\makebox(0,0)[tl]{\strut{}{
    \begin{minipage}[h]{318.389136pt}
\textcolor{inkcol1}{\LARGE{$g_i$}}\\
\end{minipage}}}}}%

 \end{picture}
\endgroup